\documentclass[12pt]{article}
\usepackage[utf8]{inputenc}
\usepackage[a4paper, total={6in, 8in}, left=1.134in]{geometry}
\usepackage{amsmath}
\usepackage{amssymb}
\usepackage{amsthm}
\usepackage{hyperref}
\usepackage{tikz-cd}
\hypersetup{colorlinks=true,linkcolor=black,citecolor=black}

\newtheorem{theorem}{Theorem}[section]
\newtheorem{lemma}[theorem]{Lemma}
\newtheorem{proposition}[theorem]{Proposition}
\newtheorem{corollary}[theorem]{Corollary}

\newtheorem{remark}[theorem]{Remark}

\numberwithin{equation}{section}

\newcommand{\R}{\mathbb{R}}
\newcommand{\Z}{\mathbb{Z}}
\newcommand{\Q}{\mathbb{Q}}

\renewcommand{\P}{\mathbb{P}}
\newcommand{\CP}{\mathbb{CP}}
\newcommand{\OO}{\mathcal{O}}
\renewcommand{\d}{\partial}
\newcommand{\dbar}{\overline{\partial}}

\newcommand{\e}{\varepsilon}

\DeclareMathOperator{\Ric}{Ric}

\DeclareMathOperator{\tr}{tr}
\DeclareMathOperator{\Kah}{Kah}

\DeclareMathOperator{\Vol}{Vol}
\DeclareMathOperator{\Real}{Re}

\DeclareMathOperator{\Diam}{Diam}
\DeclareMathOperator{\PSH}{PSH}

\newcommand{\oo}[1]{\overline{#1}}
\newcommand{\parabolic}[1]{\left(\frac{\d}{\d t} - \Delta_{#1}\right)}

\newcommand{\norm}[1]{\left\|#1\right\|}
\newcommand{\ip}[1]{\left\langle#1\right\rangle}
\newcommand{\comment}[1]{}

\title{Fano Fibrations and Twisted K\"ahler-Einstein Metrics II: The K\"ahler-Ricci Flow}
\author{Alexander Bednarek\\The University of Sydney\\\texttt{alexander.bednarek@sydney.edu.au}}
\date{\today}

\begin{document}

\maketitle

\begin{abstract}
\noindent This is the second of two papers studying both the geometric structure of Fano fibrations and the application to K\"ahler-Ricci flows developing a singularity in finite time. We assume that the K\"ahler-Ricci flow on a compact K\"ahler manifold has a rational initial metric and develops a singularity in finite time such that the manifold admits a Fano fibration structure. Moreover, it is assumed that the volume form of the flow collapses uniformly at the rate of $C^{-1}(T-t)^{n-m} \Omega \leq \omega(t)^n\leq C(T-t)^{n-m}\Omega$. Under this setting, a diameter bound is obtained in any compact set away from singular fibres and the diameter of the fibres is proven to collapse at the optimal rate $\sqrt{T-t}$. Furthermore, several precise $C^0$-estimates are proven for the potential of the complex Monge-Ampere flow which involve the potentials of singular twisted K\"ahler-Einstein metrics on the base variety from part I. Finally, in the case of K\"ahler-Einstein Fano fibres, we deduce Type I scalar curvature in any compact set away from singular fibres and globally for a submersion.
\end{abstract}

\section{Introduction}
The Analytic Minimal Model Program initiated by Song and Tian, \cite{ST07,ST12,BEG13,ZZ20,JST23}, conjectures that the K\"ahler-Ricci flow can be used as an analytic method to provide a classification for compact K\"ahler manifolds. In particular, the K\"ahler-Ricci flow is naturally expected to perform finitely many divisorial contractions, small contractions or flips until the given manifold becomes either a minimal model or Mori fiber space, that is, a Fano fibration over an analytic variety of a strictly smaller dimension, \cite{BEG13,JST23}.\\
In particular, if we take $(X,\omega_0)$ to be a compact K\"ahler manifold then there exists a $1$-parameter family of K\"ahler metrics $(\omega(t))_{t\in [0,T)}$ up to some maximal time $T \leq \infty$ that satisfies the following normalization of the K\"ahler-Ricci flow
\begin{equation*}
    \frac{\d \omega(t)}{\d t} = -\Ric (\omega(t)) - \omega(t), \quad \omega(0) = \omega_0.
\end{equation*}
If the canonical line bundle $K_X$ is ample then the flow exists for all time and, for any initial metric, will converge in the smooth topology to the unique K\"ahler-Einstein metric in the class $-2 \pi c_1(X)$, \cite{C85}. If $K_X$ is only nef then the flow develops a singularity at infinite time \cite{TZ06}. Here $X$ is a minimal model and, assuming the abundance conjecture holds, $K_X$ is semi-ample and for a sufficiently large $k \in \Z_{>0}$ generates a proper, holomorphic map with connected fibres 
\begin{equation*}
    f: X \to B \subseteq \P H^0(X,kK_X)
\end{equation*}
where $B$ is an irreducible, analytic variety and $f^* \OO(1) = K_X^{\otimes k}$. This is a Calabi-Yau fibration and it is known due the work of many, \cite{ST07,ST12,TWY18,T19,HLT24}, that the K\"ahler-Ricci flow converges in the smooth local topology away from singular fibres to a singular twisted K\"ahler-Einstein metric $\omega_{B}$ on $B$ which satisfies
\begin{equation*}
    \Ric \omega_B = -\omega_B + \omega_{WP}
\end{equation*}
on the regular part of the variety and excluding any critical points of the map $f$, and where $\omega_{WP}$ denotes the Weil-Petersson metric.\\
Finally, assume $K_X$ is not nef so that the flow encounters a singularity in a finite time, i.e. $T < \infty$, \cite{TZ06}. If there exists an ample $\Q$-line bundle $L$ such that $[\omega_0] = 2 \pi c_1(L)$, one can again generate a proper, holomorphic map with connected fibres
\begin{equation*}
    f: X \to B \subseteq \CP^N
\end{equation*}
to some irreducible analytic variety $B$ where the generic fibre is a Fano manifold. Let us denote $S = f^{-1}(S')$ where $S'$ constitutes the singular part of the variety $B$ and critical points of the map $f$. There are several cases to consider depending on the dimension $m = \dim B$ of the base analytic variety $B$, \cite{JST23, ZZ20}.\\
If $m = 0$ then the entire manifold $X$ is Fano and will become extinct, i.e. collapse to a point in finite time, and the Hamilton-Tian conjecture states that after performing a Type I rescaling in time, the flow will converge in the Gromov-Hausdorff topology to a $\Q$-Fano variety admitting a singular K\"ahler-Ricci soliton metric and, furthermore, in the smooth local topology on the regular part of the variety. This is confirmed in dimensions $\dim X \leq 3$ by Tian, Zhang, \cite{TZ16}, and in general dimension by the work of \cite{CW16,WZ21}. We also note that in this setting Perelman proved that the scalar curvature along the K\"ahler-Ricci flow is Type I \cite{ST08}.\\
If $0 < m < n$ then $X$ is expected to collapse onto the base of the fibration and the fibres of the maps should collapse to points \cite{JST23}. More precisely, the flow should converge in the smooth local topology away from singular fibres to the pullback of a singular K\"ahler metric $\omega_B$ on $B$. The generalized Hamilton-Tian conjecture states that the Type I blow-up rescaling of the restriction of the flow to any fibre should again converge in the Gromov-Hausdorff topology to a $\Q$-Fano variety admitting a singular K\"ahler-Ricci soliton metric and this convergence should extend to the smooth local topology on the regular part of the variety. This is the case we will consider for the remainder of the paper.\\
Finally, if $m = n$ then the flow is non-collapsed and expected to perform divisorial contractions, small contractions and flips and similarly converge in the smooth local topology to some singular K\"ahler metric outside of a subvariety \cite{JST23, ZZ20}. Additionally, in this case and the other cases mentioned, it is generally expected that the flow should globally converge in the Gromov-Hausdorff topology to a metric space whose structure is related to the singular K\"ahler metrics.\\

\noindent Now, in order to extract several important results for the case of interest, we take the following assumption on the collapsing of the volume form along the K\"ahler-Ricci flow. We assume that there exists constants $U_X > 0$ and, for any compact set $K \subseteq X \backslash S$, $C_K > 0$, such that
\begin{equation}\tag{VFC}
    C_K^{-1} (T-t)^{n-m} \Omega \leq \omega(t)^n \leq C_K(T-t)^{n-m} \Omega \quad \text{on } K \times [0,T),
\end{equation}
and
\begin{equation*}
    \omega(t)^n \leq U_X (T-t)^{n-m} \Omega \quad \text{on } X \times [0,T),
\end{equation*}
for some volume form $\Omega$ on $X$. This rate of collapsing is true for the volume $\Vol(X,\omega(t))$ along the flow and both the upper and lower collapsing behaviour is expected to be true, in general, globally on $X$. So we take the described, weaker condition and comment that it is useful to distinguish the quantities $C_K$ and $U_X$ for the sake of certain estimates.\\

\noindent We have the following result.

\begin{theorem}
Assume that the K\"ahler-Ricci flow satisfies (\ref{VFC}). For any compact set $K' \subseteq B\backslash S'$, there exists a positive constant $C = C(n,m,\omega_0,f^*\eta,T, K', C_{f^{-1}(K')})$ such that for all $(b,t) \in K' \times[0,T)$ we have
\begin{equation*}
     C^{-1} \sqrt{T-t} \leq \Diam(X_b, \omega(t)|_{X_b}) \leq C \sqrt{T-t}.
\end{equation*}
Moreover, for any path-connected, compact set $K \subseteq X \backslash S$, there exists a positive constant $C = C(n,m,\omega_0,f^*\eta,T,K,C_K)$ such that for all $t\in [0,T)$ we have
\begin{equation*}
    C^{-1}\Diam(f(K),\eta|_{f(K)}) \leq \Diam(K, \omega(t)|_K) \leq C,
\end{equation*}
where $\eta$ denotes a multiple of the Fubini-Study metric.
\end{theorem}

\noindent This is a natural extension of the work of Fong, \cite{F15}, who proves this result in the case where $f: X \to B$ is a submersion under the much stronger assumption that the Ricci curvature is uniformly bounded above, $\Ric \omega(t) \leq C \omega_0$. Furthermore, we note that recently Guo, Phong, Sturm, Song proved the upper bound $\Diam(X, \omega(t)) \leq C$ in full generality for finite time singularities of the K\"ahler-Ricci flow whose limiting cohomology class is big \cite{GPSS24}. Also, see \cite{FZ16,S21,SSW12,ZZ23} for more results concerning diameters.\\

\noindent In \cite{ZZ20}, Zhang and Zhang study the finite time singularities of the continuity method in the setting where $0 < m < n$ and construct a singular K\"ahler metric $\omega_B$ such that the continuity method converges to $f^*\omega_B$ in the $C^0_{loc}(X \backslash S)$-topology. It is natural to consider the convergence of the K\"ahler-Ricci flow to this same singular K\"ahler metric and that is the main motivation of this paper. In the first paper \cite{B25II} of this two part series, we constructed a $(1,1)$-form $\omega_{WP,\lambda}$ motivated by the Weil-Petersson, such that the singular metric $\omega_B$ of Zhang and Zhang satisfies 
\begin{equation*}
    \Ric \omega_B = -\omega_B - \lambda \eta + \omega_{WP,\lambda}
\end{equation*}
on $B \backslash S'$, where $\eta$ is a multiple of the Fubini-Study metric and $\lambda > 0$ is a constant. However, we also construct another singular K\"ahler metric, $\omega_B'$, that satisfies the following twisted K\"ahler-Einstein equation
\begin{equation*}
    \Ric \omega_B' = -\omega_B' + \omega_{WP,\lambda}
\end{equation*}
on $B \backslash S'$. This is another candidate for the convergence of the K\"ahler-Ricci flow. From \cite{ZZ20, B25II} we note these metrics, by definition, are equal to $\omega_B = \eta + i\d\dbar \rho_B$ and $\omega_B' = \frac{1}{1-e^{-T}} \eta + i\d\dbar \rho_B'$ for some functions $\rho_B, \rho_B' \in C^0(B) \cap C^{\infty}(B \backslash S')$. Additionally, the Fano fibration admits a $(1,1)$-form $\omega_{SPR}$ on $X \backslash S$, called the semi-prescribed Ricci curvature form, such that $\omega_{SPR} = \omega_0 + i\d\dbar \rho_{SPR}$ for some $\rho_{SPR} \in C^{\infty}(X \backslash S)$ is a K\"ahler metric on each smooth fibre with Ricci curvature prescribed by the initial metric.\\
Furthermore, it is of interest to consider the case where the fibration generically admits a smoothly varying of K\"ahler-Einstein metrics. To be clear, we assume there exists a $p > 1$ and $\rho_{SKE} \in C^{\infty}(X\backslash S)$ such that $e^{-\lambda\rho_{SKE}} \in L^p(X, \Omega)$ and $\rho_{SKE}$ can have a logarithmic pole of a small order from above along $S$. Additionally,
the semi K\"ahler-Einstein form $\omega_{SKE} = \omega_0 + i\d\dbar \rho_{SKE}$ is required to a K\"ahler-Einstein metric on each smooth fibre. Let us call this condition (SKE) \cite{B25II}.\\
We reduce the K\"ahler-Ricci flow to the following complex Monge-Ampere flow
\begin{equation*}
    \frac{\d \varphi}{\d t} = \log\frac{E(t)^{-(n-m)} \omega(t)^n}{\Omega} - \varphi, \quad \varphi(0) = 0,
\end{equation*}
where $\Omega$ is a volume form and $E(t) = \frac{e^{-t}-e^{-T}}{1-e^{-T}}$ for convenience. Then, given the data above, we have the following results.

\begin{theorem}
Assume the K\"ahler-Ricci flow satisfies (\ref{VFC}). For any compact set $K \subseteq X \backslash S$ there exists a real function $h_K = h_K(n,m,T,\omega_0,f^*\eta,\rho_{SPR},f^*\rho_B,U_X)$, where $h_K:[0,T) \to [0,\infty)$ is continuous, strictly decreasing and satisfies $h_K \to 0$ as $t \to T$ such that, on $K \times [0,T)$, we have
\begin{equation*}
    \lambda \inf_{X \backslash S} (\rho_{SPR} - f^*\rho_B) - h_K(t) \leq \varphi(t) - f^*\rho_B \leq h_K(t) + \lambda \sup_{X \backslash S} (\rho_{SPR} - f^*\rho_B).
\end{equation*}
Moreover, there exists a $h_K:[0,T) \to [0,\infty)$, where $h_K = h_K(n,m,T,\omega_0,f^*\eta,\rho_{SPR},f^*\rho'_B,U_X)$, such that on $K \times [0,T)$
\begin{equation*}
    \lambda \inf_{X \backslash S} \rho_{SPR} - h_K(t) \leq \varphi(t) - (1-e^{-T}) f^*\rho_B' \leq h_K(t) + \lambda \sup_{X \backslash S} \rho_{SPR}.
\end{equation*}
Furthermore, if (\ref{SKE}) holds then there exists a $h_K: [0,T) \to [0,\infty)$, where $h_K = h_K(n,m,T,\omega_0,f^*\eta,\rho_{SKE},f^*\rho_B, U_X)$, such that on $K \times [0,T)$
\begin{equation*}
    \lambda \inf_{X \backslash S} (- f^*\rho_B) - h_K(t) \leq \varphi(t) - f^*\rho_B \leq h_K(t) + \lambda \sup_{X \backslash S} (- f^*\rho_B).
\end{equation*}
Finally, if (\ref{SKE}) holds then there exists a $h_K: [0,T) \to [0,\infty)$, where $h_K = h_K(n,m,T,\omega_0,f^*\eta,\rho_{SKE},f^*\rho'_B, U_X)$, such that on $K \times [0,T)$
\begin{equation*}
     - h_K(t) \leq \varphi(t) - (1-e^{-T})f^*\rho'_B \leq h_K(t).
\end{equation*}
\end{theorem}

\noindent This is an extension of the result \cite[Prop. 5.4]{ST12} of Song and Tian to the case of finite time singularities. For instance, if $\rho_{SPR} = 0$ on $X \backslash S$ then one can say $\omega(t) \to (1-e^{-T})f^*\omega_B'$ in the $C^0_{loc}(X\backslash S)$-sense of potentials. However, $\rho_{SPR} = 0$ is quite strong as it implies $\omega_{SPR} = \omega_{0}$ and thus the initial metric is K\"ahler-Einstein along fibres. Hence, it is more natural to take the assumption (\ref{SKE}). One can be hopeful then that, assuming (\ref{SKE}) holds, we have $\omega(t) \to (1-e^{-T}) f^*\omega_B'$ in the $C^\infty_{loc}(X \backslash S)$-topology.\\

\noindent Additionally, when $B$ is a smooth K\"ahler manifold and the map $f:X \to B$ is a submersion, the dependence on the volume form collapsing condition, (\ref{VFC}), can be removed and the convergence rate $h_X(t)$ improved to $(T-t)$. If we also assume some conditions on $\omega_{SPR}$, $\omega_0$, $f^*\omega_B$, $f^*\eta$, we can deduce convergence in the $C^0(X)$-topology of $\omega(t)$ the twisted K\"ahler-Einstein metrics in the sense of their potentials.

\begin{theorem}
Assume that the map $f: X \to B$ is a submersion. If $\omega_{SPR} - \omega_0 + f^*\omega_B - f^*\eta = 0$ on $X$, then there exists a constant $C = C(n,m,T,\omega_0,f^*\eta,\rho_{SPR},f^*\rho_B) > 0$ such that on $X \times [0,T)$ we have
\begin{equation*}
    \left|\varphi(t) - f^*\rho_B\right| \leq C(T-t).
\end{equation*}
If $\omega_{SPR} = \omega_{0}$ on $X$, then there exists a constant $C = C(n,m,T,\omega_0,f^*\eta, \rho_{SPR}, f^*\rho_B') > 0$ such that on $X \times [0,T)$ we have
\begin{equation*}
    \left|\varphi(t) - (1-e^{-T}) f^*\rho_B'\right| \leq C(T-t).
\end{equation*}
If (\ref{SKE}) holds and $f^*\omega_{B} = f^*\eta$ on $X$, then there exists a positive constant $C = C(n,m,T,\omega_0,f^*\eta,\rho_{SKE}, f^*\rho_B) > 0$ such that on $X \times [0,T)$ we have
\begin{equation*}
    \left|\varphi(t) - f^*\rho_B\right| \leq C(T-t).
\end{equation*}
If (\ref{SKE}) holds then there exists a constant $C = C(n,m,T,\omega_0,f^*\eta, \rho_{SKE},f^*\rho_B') > 0$ such that on $X \times [0,T)$ we have
\begin{equation*}
    \left|\varphi(t) - (1-e^{-T}) f^*\rho_B'\right| \leq C(T-t).
\end{equation*}
\end{theorem}

\noindent Finally, in regards to the curvature behaviour encountered by finite time singularities, little is known concerning the Ricci and Riemannian curvatures, \cite{B25I,CHM25,MT25}, but there are several results concerning the scalar curvature which we now mention.\\
It was previously known due to Zhang, \cite{Z10}, that in this setting where $0 < m < n$, the K\"ahler-Ricci flow has scalar curvature bounded on $X \times [0,T)$ by
\begin{equation*}
    R_{\omega(t)} \leq \frac{C}{(T-t)^2},
\end{equation*}
and due to Perelman, \cite{ST08}, that for Fano manifolds, i.e. $m = 0$, the K\"ahler-Ricci flow has Type I scalar curvature
\begin{equation*}
    R_{\omega(t)} \leq \frac{C}{T-t}
\end{equation*}
on $X \times [0,T)$. Furthermore, a more technical result of Hallgren, Jian, Song, Tian, \cite{JST23,HJST23}, states that for all $(x,t) \in X \times [0,T)$,
\begin{equation*}
    (T-t)R(x,t) \leq C\left(1 + \frac{d^2_{\omega(t)}(x,p(t))}{T-t}\right)
\end{equation*}
where $p(t)$ is a point at time $t$ defined to be minimizing a particular function on $X$.\\

\noindent We present the following improvement and extension of these above results.

\begin{theorem}
Assume that the K\"ahler-Ricci flow satisfies (\ref{VFC}). If $\rho_{SPR} = 0$ on $X \backslash S$ or (\ref{SKE}) holds, then for any compact set $K \subseteq X \backslash S$ where $h_K$ is Lipschitz, there exist a constant $C = C(n,m,T,\omega_0,f^*\eta,\rho_{SPR},f^*\rho_B',K,C_K,U_X) > 0$ or a constant $C = C(n,m,T,\omega_0,f^*\eta,\rho_{SKE},f^*\rho_B',K,C_K,U_X) > 0$, respectively, such that on $K \times [0,T)$,
\begin{equation*}
    R_{\omega(t)} \leq \frac{C}{T-t}.
\end{equation*}
On the other hand, assume that $f:X \to B$ is a submersion and the K\"ahler-Ricci flow satisfies (\ref{VFC}). If $\omega_{SPR} = \omega_{0}$ on $X$ or if (\ref{SKE}) holds, then there exists a $C = C(n,m,T,\omega_0, f^*\eta,\rho_{SPR},f^*\rho_B',C_X) > 0$ or $C = C(n,m,T,\omega_0,f^*\eta,\rho_{SKE},f^*\rho_B',C_X) > 0$, respectively, such that on $X \times [0,T)$ we have
\begin{equation*}
    R_{\omega(t)} \leq \frac{C}{T-t}.
\end{equation*}
\end{theorem}

\noindent \textbf{Acknowledgements.} \textit{The author's gratitude goes to his supervisors Zhou Zhang and Haotian Wu for all their support. The interests of and discussions with Tiernan Cartwright and Xiaohua Zhu are also extremely appreciated. The author would like to thank Wangjian Jian, Zhenlei Zhang and Gang Tian, for extending an invitation to visit BICMR and for many discussions, during which the idea for this paper was conceived. The research visit was supported by the Chinese Academy of Sciences and the Beijing International Center for Mathematical Research.}

\section{The K\"ahler-Ricci Flow}

Let $(X,\omega_0)$ be a compact K\"ahler manifold of complex dimension $n = \dim X$. We consider the following normalization of the K\"ahler-Ricci flow,
\begin{equation}\label{KRF}\tag{KRF}
    \frac{\d \omega(t)}{\d t} = -\Ric(\omega(t)) - \omega(t), \quad \omega(0) = \omega_0.
\end{equation}
The (\ref{KRF}) exists for a maximal time $T \in (0,\infty]$ determined by the following cohomological condition
\begin{equation*}
    T = \sup \{ t \in [0,\infty): e^{-t} [\omega_0] - (1-e^{-t}) 2\pi c_1(X) \in \Kah(X)\},
\end{equation*}
where $\Kah(X)$ denotes the K\"ahler cone in $H^{1,1}(X,\R)$, \cite{TZ06}.\\
We assume $T < \infty$ and the initial metric is rational, i.e. $[\omega_0] \in H^{1,1}(X,\Q)$ and there exists an ample $\Q$-line bundle $L$ such that $2\pi c_1(L) = [\omega_0]$. By Kawamata's rationality theorem, 
\begin{equation*}
    D := e^{-T} L + (1-e^{-T}) K_X
\end{equation*}
is a rational line bundle, i.e. $e^{-T} \in \Q$, and Kawamata's basepoint free theorem implies that $D$ is a semi-ample line bundle. That is, for a sufficiently large $k \in \Z_{>0}$, $kD$ generates a proper, holomorphic, surjective map with connected fibres
\begin{equation*}
    f: X \to B \subseteq \P H^0(X,kD)
\end{equation*}
where $B = f(X)$ denotes the image of $f$ and is an irreducible, normal, analytic variety. Additionally, $f^* \OO_B(1) = kD$ so that if we set $\eta = \frac{1}{k}\omega_{FS}|_{B}$, where $\omega_{FS}$ is the Fubini-Study metric on $\CP^N$ and $k$ is fixed to be the smallest $k \in \Z_{>0}$ which generates the map $f$ and satisfies $k' = k(1-e^{-T}) \in \Z_{>0}$, then
\begin{equation*}
[\omega(T)] = 2\pi c_1(D) = \frac{2\pi}{k} f^*c_1( \OO_B(1)) = [f^*\eta],
\end{equation*}
where we take the $[\omega(T)]$ to be the limiting class of the flow, $e^{-T} [\omega_0] - (1-e^{-T})2\pi c_1(X)$.\\
Let us denote the singular set of $B$ together with the critical points of $f$ by $S'$ and set $S = f^{-1}(S')$. Thus, $f: X \backslash S \to B \backslash S'$ is a submersion and $X_b = f^{-1}(b)$ is a K\"ahler manifold equipped with K\"ahler metric $\omega_{0,b} = \omega_0|_{X_b}$ for all $b \in B \backslash S'$. By restricting the cohomological equation $[f^*\eta] = [\omega(T)]$ to any non-singular fibre $X_b$, we see that the fibres are also Fano, i.e. $2\pi c_1(X_b) = \lambda [\omega_{0,b}]$ where $\lambda =\frac{e^{-T}}{1-e^{-T}} > 0$ is a positive constant. Finally, we let the Iitaka dimension be denoted by $m = \dim B$ and we assume that the fibres are of intermediate dimension, i.e. $0 < m < n$.\\

\noindent Let us reduce the K\"ahler-Ricci flow to a complex Monge-Ampere flow. As the cohomology class $[\omega(t)]$ of the K\"ahler-Ricci flow evolves according to
\begin{equation*}
    [\omega(t)] = \frac{e^{-t}-e^{-T}}{1-e^{-T}}[\omega_0] + \frac{1-e^{-t}}{1-e^{-T}} [\omega(T)],
\end{equation*}
we consider the reference form $\omega_{REF}(t) \in [\omega(t)]$ defined by
\begin{equation*}
    \omega_{REF}(t) = \frac{e^{-t}-e^{-T}}{1-e^{-T}} \omega_0 + \frac{1-e^{-t}}{1-e^{-T}} f^* \eta.
\end{equation*}
Note $\omega_{REF}(t)$ is in fact a K\"ahler metric for every $t \in [0,T)$. By the $\d\dbar$-lemma there exists a smooth function $\varphi: X \times [0,T) \to \R$ such that
\begin{equation*}
    \omega(t) = \omega_{REF}(t) + i\d\dbar \varphi(t).
\end{equation*}
Moreover, as $[f^*\eta] = [\omega(T)] = e^{-T}[\omega_0] - (1-e^{-T}) 2\pi c_1(X)$, we have
\begin{equation*}
    \chi := \lambda \omega_0 -\frac{1}{1-e^{-T}} f^*\eta\in 2\pi c_1(X)
\end{equation*}
so that we can choose a volume form $\Omega$ on $X$ such that $\chi = \Ric \Omega$. We fix $\Omega$ to be unique by imposing
\begin{equation*}
    \int_X \Omega = \int_X {n\choose m} \omega_0^{n-m} \wedge f^*\eta^m.
\end{equation*}
It follows that the (\ref{KRF}) reduces to the following complex Monge-Ampere flow
\begin{equation}\label{CMA}\tag{CMA}
    \frac{\d \varphi}{\d t} = \log\frac{E(t)^{-(n-m)}\omega(t)^n}{\Omega}-\varphi, \quad \varphi(0) = 0,
\end{equation}
where we set $E(t) = \frac{e^{-t}-e^{-T}}{1-e^{-T}}$ for convenience. We define
\begin{equation*}
u = (1-e^{t-T}) \frac{\d \varphi}{\d t} + \varphi
\end{equation*}
to be the Ricci potential since one can show, \cite{Z10}, the Ricci curvature decomposes in terms of the Ricci potential according to
\begin{equation*}
    (1-e^{t-T})\Ric (\omega(t)) = e^{t-T}\omega(t) - f^*\eta - i\d\dbar  u.
\end{equation*}
Additionally, the scalar curvature is given by
\begin{equation}\label{scalar_curvature_formula}
    (1-e^{t-T}) R_{\omega(t)} = e^{t-T}n - \tr_{\omega(t)} f^*\eta - \Delta_{\omega(t)} u.
\end{equation}

\section{Consequences of Uniform Volume Form Collapsing}

\noindent First, an observation. 

\begin{proposition}
There exists a constant $C = C(n,m,\omega_0,f^*\eta, T)> 0$ such that for all $t \in [0,T)$,
\begin{equation*}
    C^{-1}E(t)^{n-m} \leq \Vol(X, \omega(t)) \leq CE(t)^{n-m}.
\end{equation*}
In fact,
\begin{equation*}
    \lim_{t \to T} E(t)^{-(n-m)} \Vol(X, \omega(t)) = \Vol(X, \Omega_{\omega_0,\eta}),
\end{equation*}
where we use the notation
\begin{equation*}
    \Omega_{\omega_0,\eta} = {n \choose m} \omega_0^{n-m} \wedge f^*\eta^m.
\end{equation*}
\end{proposition}

\begin{proof}
This is an immediate consequence of the cohomological evolution of the K\"ahler-Ricci flow. Applying $\eta^k = 0$ for all $m < k \leq n$, we have
\begin{align*}
    E(t)^{-(n-m)} \Vol(X, \omega(t))
    &= \int_X E(t)^{-(n-m)}\omega_{REF}^n\\ 
    &= \int_{X} \sum_{k=0}^m {n \choose k} E(t)^{m-k} (1-E(t))^{k} \omega_0^{n-k} \wedge f^* \eta^k
\end{align*}
and the limit as $t \to T$ is equal to $\int_X \Omega_{\omega_0, \eta}$.
\end{proof}

\noindent This is well-known, i.e., that the volume along the flow collapses at the rate $E(t)^{n-m}$. As a consequence, one naturally expects that, in general, the volume form $\omega(t)^n$ collapses uniformly at the $E(t)^{n-m}$-rate globally on $X$. It is known for the continuity method that this is true \cite[Lem. 3.2]{ZZ20}, i.e. $\omega(t)^n \sim (T-t)^{n-m} \Omega$, but this behaviour equivalent to the $C^0$-bound of the potential function $\varphi$, which is unfortunately not the case for the K\"ahler-Ricci flow. On the other hand, when the K\"ahler-Ricci flow is run on a manifold with semi-ample canonical line bundle $K_X$, the flow develops a singularity at infinite time and $X$ admits a Calabi-Yau fibration structure. It is known that the volume form collapses at a rate of $\omega(t)^n \sim e^{-(n-m)t}\Omega$ in this setting, \cite{ST16},\cite[Lem. 2.3]{FZ12}, which would correspond to letting $T \to\infty$ in $E(t)$.\\

\noindent We take a slightly weaker version of the uniform collapsing of the volume form as an assumption. That is, we assume there exists a positive constant $U_X > 0$ and for any compact set $K \subseteq X \backslash S$ a positive constant $C_K > 0$, such that
\begin{align}\tag{VFC}\label{VFC}
    C_K^{-1}E(t)^{n-m}\Omega \leq \omega(t)^{n} \leq C_KE(t)^{n-m}\Omega \quad \text{on} \quad K \times [0,T),\\
    \omega(t)^n \leq U_XE(t)^{n-m} \Omega \quad \text{on} \quad X \times [0,T).\notag
\end{align}
Clearly these assumptions are independent of the choice of volume form $\Omega$. Moreover, if $f: X \to B$ is a submersion then we can simply set $K = X$. We also comment that it is useful to distinguish the local lower bounds and global upper bound with different constants $C_K$, $U_X$.

\begin{remark}
We compare this assumption with the work of Fong \cite{F15}. Fong considers the case where the map $f: X \to B$ is a submersion and assumes that the Ricci curvature is uniformly bounded above along the flow, i.e. $\Ric \omega(t) \leq C \omega_0$ on $X \times [0,T)$ for some $C > 0$. This is a strong assumption but allows Fong to find an upper bound on $\frac{\d \varphi}{\d t}$ and $U_X$ and, subsequently, an optimal bound on the collapsing rate of the diameters of the fibres. We can achieve similar results with our much weaker assumption (\ref{VFC}) and extend them to the general case where $X$ has singular fibres.
\end{remark}

\subsection{$C^0$-Estimates}

\begin{proposition}\label{submersion_C0_varphi_estimate}
There exists a constant $C = C(n,m,\omega_0,f^*\eta,T) > 0$ such that on $X \times [0,T)$ we have
\begin{equation*}
    \varphi \leq C.
\end{equation*}
Furthermore, if $f: X \to B$ is a submersion of compact K\"ahler manifolds then we also have
\begin{equation*}
    \varphi \geq -C,
\end{equation*}
on $X \times [0,T)$ for some $C = C(n,m,\omega_0,f^*\eta, T) > 0$.
\end{proposition}

\begin{proof}
Note that
\begin{equation*}
    E(t)^{-(n-m)}\omega_{REF}^n = \sum_{k=0}^m {n \choose k} E(t)^{m-k} (1-E(t))^{k} \omega_0^{n-k} \wedge f^* \eta^k
\end{equation*}
satisfies
\begin{equation*}
    C^{-1} \Omega_{\omega_0,\eta} \leq E(t)^{-(n-m)} \omega_{REF}^n \leq C\Omega
\end{equation*}
on $X \times [T/2, T)$, since there exists a $C > 0$ such that for all $0 \leq k \leq m$
\begin{equation*}
    0 \leq \omega_0^{n-k} \wedge f^* \eta ^k \leq C\Omega
\end{equation*}
and we can bound the term $(1-E(t))^m$ below on $[T/2,T)$.\\
Therefore, by applying the maximum principle to (\ref{CMA}) we find that at a maximal point $(x^*,t^*)$ of $\varphi$,
\begin{equation*}
    \varphi(x^*,t^*) \leq \log \frac{E(t^*)^{-(n-m)}\omega^n_{REF}(x^*,t^*)}{\Omega(x^*)} \leq C.
\end{equation*}
Thus $\varphi \leq C$ on $X \times [0,T)$.\\
Additionally, if $f: X \to B$ is a submersion then $\Omega_{\omega_0,\eta}$ is a global volume form on $X$. It follows that at a minimal point $(x_*,t_*)$ of $\varphi$
\begin{equation*}
     \varphi(x_*,t_*) \geq \log \frac{E(t_*)^{-(n-m)} \omega_{REF}(x_*,t_*)^{n}}{\Omega(x_*)} \geq \log \frac{C^{-1}\Omega_{\omega_0,\eta}(x_*)}{\Omega(x_*)} \geq -C,
\end{equation*}
and thus $\varphi \geq -C$ on $X \times [T/2,T)$. This yields the desired result since $X \times [0,T/2]$ is compact.
\end{proof}

\begin{proposition}
We have the following evolution equations
\begin{align*}
    \parabolic{} \left(\frac{\d \varphi}{\d t} + \varphi\right) &= -n + \frac{n-m}{1-e^{t-T}} - \lambda \tr_{\omega(t)} \omega_0 + \frac{1}{1-e^{-T}} \tr_{\omega(t)} f^* \eta,\\
    \parabolic{} u &= -m + \tr_{\omega(t)} f^* \eta.
\end{align*}
\end{proposition}

\begin{proof}
Applying $\frac{\d}{\d t}$ to (\ref{CMA}) yields
\begin{equation}\label{time_derivative_eqn}
    \frac{\d}{\d t} \left( \frac{\d \varphi}{\d t} + \varphi\right)
    = \frac{n-m}{1-e^{t-T}} - R - n,
\end{equation}
whereas applying $i\d\dbar$ to (\ref{CMA}) yields
\begin{equation*}
    i\d\dbar\left(\frac{\d \varphi}{\d t} + \varphi\right) = -\Ric\omega + \lambda \omega_0 - \frac{1}{1-e^{-T}} f^* \eta.
\end{equation*}
The first evolution equation follows. On the other hand, observe that
\begin{equation*}
    \parabolic{} u = (1-e^{t-T}) \parabolic{} \left(\frac{\d \varphi}{\d t} + \varphi\right) - e^{t-T} \Delta \varphi.
\end{equation*}
The second evolution equation then follows from the first and by taking the trace of $\omega(t) = \omega_{REF}(t) + i\d\dbar \varphi$.
\end{proof}

\begin{proposition}
There exists a constant $C = C(n,m,\omega_0,f^*\eta, T) > 0$ such that on $X \times [0,T)$ we have
\begin{equation*}
    |u| \leq C.
\end{equation*}
\end{proposition}

\begin{proof}
From the evolution equation for $u$, it follows that
\begin{equation*}
    \parabolic{} u \geq -m
\end{equation*}
and the maximum principle implies $-C \leq u$. On the other hand, by the uniform lower bound on scalar curvature, it follows from equation (\ref{time_derivative_eqn}) that
\begin{equation*}
    \frac{\d }{\d t} \left( \frac{\d \varphi}{\d t} + \varphi + (n-m) \log (e^{-t}-e^{-T}) - Ct\right) \leq 0.
\end{equation*}
Integration yields
\begin{equation*}
    \frac{\d \varphi}{\d t} + \varphi + (n-m) \log (e^{-t}-e^{-T}) - Ct \leq C
\end{equation*}
and the desired result, $u \leq C$, will quickly follow.
\end{proof}

\begin{remark}
We note that this bound on $\frac{\d \varphi}{\d t}$ in combination with that of $\varphi$ only implies $\omega(t)^n \leq C\Omega$ which is already known and not desired. Also, the lower bound for $u$ in combination with the upper bound of $\varphi$ only implies $\frac{\d \varphi}{\d t} \geq -\frac{C}{1-e^{t-T}}$.
\end{remark}

\begin{lemma}[Parabolic Schwarz Lemma] For any K\"ahler metric $\theta$ on $\CP^N$, there exists a $C = C(n,m,\omega_0, f^*\eta, f^*\theta, T)> 0$ such that on $X \times [0,T)$ we have
\begin{equation*}
    0 \leq \tr_{\omega(t)} f^* \theta \leq C.
\end{equation*}
Moreover, for any smooth $(1,1)$-form $\alpha$ on $\CP^N$, there exists a positive constant $C = C(n,m,\omega_0,f^*\eta,f^*\alpha,T) > 0$ such that on $X \times [0,T)$ we have
\begin{equation*}
    -C \leq \tr_{\omega(t)} f^* \alpha \leq C.
\end{equation*}
\end{lemma}

\begin{proof}
It is well known that the following inequality holds
\begin{equation*}
    \parabolic{} \log (\tr_{\omega(t)} f^* \eta) \leq C \tr_{\omega(t)} f^* \eta + 1
\end{equation*}
where $C$ is a constant depending on the bisectional curvature of $\eta$, \cite{ST07}. Therefore, choosing a constant $A$ sufficiently large such that $A \geq C + 1$ yields
\begin{equation*}
    \parabolic{} \left(\log (\tr_{\omega(t)} f^* \eta) - Au\right) \leq - \tr_{\omega(t)} f^* \eta +Am + 1.
\end{equation*}
The maximum principle implies that $\tr_{\omega(t)} f^* \eta \leq C$. Now, for any $(1,1)$-form $\alpha$ we can choose $A$ to be a sufficiently large constant such that $A\eta - \alpha > 0$ is a K\"ahler metric on $\CP^N$. Hence,
\begin{equation*}
    \tr_{\omega(t)} f^* \alpha \leq A \tr_{\omega(t)} f^* \eta \leq C.
\end{equation*}
The lower bound follows by the same argument.
\end{proof}

\begin{proposition}\label{dvarphi_dt_bound}
If the (\ref{KRF}) satisfies (\ref{VFC}) then there exists positive constant $C = C(n,m,\omega_0,f^*\eta,T,U_X)$ such that on $X \times [0,T)$,
\begin{equation*}
    \frac{\d \varphi}{\d t} \leq C.
\end{equation*}
Also, for any compact set $K \subseteq X \backslash S$, there exists a constant $C = C(n,m,\omega_0,f^*\eta,T,C_K) > 0$ such that
\begin{equation*}
    \left|\frac{\d \varphi}{\d t}\right| \leq C
\end{equation*}
on $K \times [0,T)$.
\end{proposition}

\begin{proof}
We note that (\ref{VFC}) implies
\begin{equation*}
    \left|\frac{\d \varphi}{\d t} + \varphi\right| \leq C
\end{equation*}
on $K \times [0,T)$ and thus
\begin{equation*}
    \left|e^{t-T}\frac{\d \varphi}{\d t}\right| \leq |u| + \left|\frac{\d \varphi}{\d t} + \varphi\right| \leq C.
\end{equation*}
Likewise, on $X \times [0,T)$, (\ref{VFC}) yields
\begin{equation*}
    \frac{\d \varphi}{\d t} + \varphi \leq C
\end{equation*}
and so
\begin{equation*}
    e^{t-T} \frac{\d \varphi}{\d t} = -u + \frac{\d \varphi}{\d t} + \varphi \leq C.
\end{equation*}
The result follows.
\end{proof}

\begin{proposition}
If the (\ref{KRF}) satisfies (\ref{VFC}) then there exists a constant $C = C(n,m,\omega_0,f^*\eta, T, U_X) > 0$ such that on $X \times [0,T)$ we have
\begin{equation*}
    \varphi \geq -C.
\end{equation*}
\end{proposition}

\begin{proof}
We note that for any $p > 1$
\begin{equation*}
    \int_{X} \left(\frac{E(t)^{n-m} e^{\frac{\d \varphi}{\d t}}+\varphi}{\Vol(X,\omega(t))}\right)^{p} \Omega \leq C
\end{equation*}
by Proposition \ref{dvarphi_dt_bound}. So applying \cite[Thm 2.5]{ST12}, (see also \cite{DP10, EGZ08}), yields
\begin{equation*}
    \sup_X \varphi - \inf_X \varphi \leq C.
\end{equation*}
Finally, 
\begin{equation*}
    e^{\sup_{X}(\varphi + \frac{\d \varphi}{\d t})} \int_X \Omega \geq \int_X e^{\varphi + \frac{\d \varphi}{\d t}} \Omega = \int_X E(t)^{-(n-m)} \omega(t)^n \geq C^{-1}
\end{equation*}
so that
\begin{equation*}
    \sup_{X} \left(\varphi + \frac{\d \varphi}{\d t}\right) \geq -C.
\end{equation*}
Therefore, we get our desired lower bound by applying Proposition \ref{dvarphi_dt_bound} again
\begin{equation*}
    \inf_X \varphi \geq -C + \sup_X \varphi \geq -C + \sup_X \left(\varphi + \frac{\d \varphi}{\d t}\right) - \sup_X \frac{\d \varphi}{\d t} \geq -C.
\end{equation*}
\end{proof}

\subsection{The Diameter Bounds}

\noindent Let  $K'$ be a compact set and $U'$ an open set such that $K' \subset \subset U' \subset \subset B\backslash S'$. As $f$ is a proper map, $f^{-1}(K')$ is a compact subset of $ X\backslash S$ and note that for every $K \subseteq X \backslash S$ we can always find a $K' \subseteq B \backslash S'$ such that $K \subseteq f^{-1}(K')$ and for $K = f^{-1}(K')$, $U = f^{-1}(U')$ we always have $K \subset \subset U \subset \subset X \backslash S$. We can choose $\chi$, \cite[Lem. 4.2]{J20}, \cite{T10}, to be a smooth cut-off function on $X$ where $0 \leq \chi \leq 1$ and
\begin{equation*}
    \chi = 0 \quad \text{ on } X\backslash U, \quad \chi = 1 \quad \text{ on } K,
\end{equation*}
and for some $C = C(K,f^*\eta) > 0$
\begin{equation*}
    -C f^*\eta \leq i\d\dbar \chi \leq C f^*\eta, \quad 0 \leq i\d \chi \wedge \dbar \chi \leq C f^*\eta.
\end{equation*}
It follows from the parabolic Schwarz lemma that there exists a positive constant $C = C(n,m,\omega_0,f^*\eta,T,K)$ such that on $X \times [0, T)$ we have
\begin{equation*}
    |\Delta_{\omega(t)} \chi | \leq C, \quad |\nabla \chi|_{\omega(t)}^2 \leq C.
\end{equation*}

\begin{proposition}\label{trace_lemma}
If the (\ref{KRF}) satisfies (\ref{VFC}) then for any K\"ahler metric $\theta$ on $X$, there exists a positive constant $C  = C(n,m,\omega_0,f^*\eta,\theta,T,K, C_K)$ such that on $K \times [0,T)$ we have
\begin{equation*}
    0 < \tr_{\omega(t)} \theta \leq \frac{C}{E(t)}.
\end{equation*}
Moreover, for any smooth $(1,1)$-form $\alpha$ on $X$, there exists a positive constant $C = C(n,m,\omega_0, f^*\eta, \alpha, T, K, C_K)$ such that on $K \times [0,T)$ we have
\begin{equation*}
    \frac{-C}{E(t)} \leq \tr_{\omega(t)} \alpha \leq \frac{C}{E(t)}.
\end{equation*}
\end{proposition}

\begin{proof}
Note that
\begin{equation*}
    \parabolic{} \log(\tr_{\omega(t)} \omega_0) \leq C \tr_{\omega(t)} \omega_0 + 1
\end{equation*}
where $C$ is a constant depending on the bisectional curvature of $\omega_0$, \cite[Thm. 3.2.6]{BEG13}. Now, let $A > 0$ be a constant to be determined, and observe that for 
\begin{equation*}
    Q_1 = \log (e^{-t}-e^{-T}) \tr_{\omega(t)} \omega_0 + A(\d_t \varphi + \varphi)
\end{equation*}
we have
\begin{align*}
    \parabolic{}Q_1 &\leq (C - \lambda A) \tr_{\omega(t)} \omega_0 + 1 - \frac{e^{-t}}{e^{-t}-e^{-T}}\\
    &+A\left(-n + \frac{n-m}{1-e^{t-T}} + \frac{1}{1-e^{-T}} \tr_{\omega(t)} f^* \eta\right)\\
    & \leq (C - \lambda A)\tr_{\omega(t)} \omega_0 + \frac{C}{e^{-t}-e^{-T}}.
\end{align*}
Next let $\Upsilon = e^{-C_1\chi^{-C_2}}$ for some constants $C_1, C_2 > 0$, so that
\begin{align*}
    \parabolic{} \Upsilon Q_1 &= \Upsilon \parabolic{} Q_1 +\left(\frac{|\nabla \Upsilon|^2}{\Upsilon} -\Delta \Upsilon\right)Q_1 - 2 \Real \ip{\nabla (\Upsilon Q_1), \frac{\oo{\nabla} \Upsilon}{\Upsilon}}.
\end{align*}
The middle term can be estimated by first calculating
\begin{align*}
    \Delta \Upsilon &= C_1 C_2 (\Delta \chi) \chi^{-C_2-1}\Upsilon-C_1 C_2(C_2+1)|\nabla \chi|^2 \chi^{-C_2-2} \Upsilon+ C_1^2 C^2_2 |\nabla \chi|^2 \chi^{-2C_2-2} \Upsilon
\end{align*}
and
\begin{align*}
    |\nabla \Upsilon|^2 &= C_1^2C_2^2 |\nabla \chi|^2 \chi^{-2C_2-2} \Upsilon^2
\end{align*}
and observing that
\begin{align*}
    \left(\frac{|\nabla \Upsilon|^2}{\Upsilon} - \Delta \Upsilon\right) Q_1 &\leq \left(2 \left|\log \sqrt{(e^{-t}-e^{-T}) \tr_{\omega(t)} \omega_0)}\right| + A|\d_t \varphi + \varphi|\right) C \chi^{-2C_2-2}\Upsilon\\ 
    &\leq C \left|\log \sqrt{(e^{-t}-e^{-T})\tr_{\omega(t)} \omega_0} \right| \chi^{-2C_2-2}\Upsilon + C \chi^{-2C_2-2}\Upsilon,
\end{align*}
where the last inequality follows from the fact that $|\d_t \varphi + \varphi| \chi^{-2C_2-2}\Upsilon$ is supported on $U$ and we can always find another compact set $K'$ such that $K \subset \subset U \subset\subset K'  \subset \subset X \backslash S$ and apply $|\d_t \varphi| \leq C(C_{K'})$.\\
All together we find
\begin{align*}
    \parabolic{} \Upsilon Q_1 &\leq \left((C - \lambda A)\tr_{\omega(t)} \omega_0 + \frac{C}{e^{-t}-e^{-T}}\right) \Upsilon\\
    &+C \left|\log \sqrt{(e^{-t}-e^{-T})\tr_{\omega(t)} \omega_0} \right| \chi^{-2C_2-2}\Upsilon + C \chi^{-2C_2-2}\Upsilon\\
    &- 2 \Real \ip{\nabla (\Upsilon Q_1), \frac{\oo{\nabla} \Upsilon}{\Upsilon}}.
\end{align*}
Let $(x^*,t^*)$ be a point at which $\Upsilon Q_1$ achieves a maximum. If $x^* \in X \backslash U$ then $(\Upsilon Q_1)(x^*,t^*) =  0$ is bounded. If $x^* \in U$ then we can also assume $(e^{-t^*} - e^{-T}) \tr_{\omega(x^*,t^*)} \omega_0(x^*) > 1$ as otherwise we get
\begin{equation*}
    (\Upsilon Q_1)(x^*,t^*) \leq 0 + A(\d_t \varphi(x^*,t^*) + \varphi(x^*,t^*))\Upsilon(x^*) \leq C
\end{equation*}
since $\Upsilon \leq C$ and $\d_t \varphi + \varphi \leq C$ on the compact set $K'$ containing $U$. Therefore,
\begin{align*}
    0 &\leq \left((C - \lambda A)\tr_{\omega(x^*,t^*)} \omega_0(x^*) + \frac{C}{e^{-t^*}-e^{-T}}\right) \Upsilon(x^*)\\
    &+C \log \left(\sqrt{(e^{-t^*}-e^{-T})\tr_{\omega(x^*,t^*)} \omega_0(x^*)}\right) \chi(x^*)^{-2C_2-2}\Upsilon(x^*) + C \chi(x^*)^{-2C_2-2}\Upsilon(x^*)\\
    &\leq \left((C - \lambda A)\tr_{\omega(x^*,t^*)} \omega_0(x^*) + \frac{C}{e^{-t^*}-e^{-T}}\right) \Upsilon(x^*)\\
    &+C\Big( (e^{-t^*}-e^{-T})\tr_{\omega(x^*,t^*)} \omega_0(x^*) \chi(x^*)^{2C_2 + 2} + \chi(x^*)^{-2C_2-2} \Big) \chi(x^*)^{-2C_2-2}\Upsilon(x^*)\\
    &+ C \chi(x^*)^{-2C_2-2}\Upsilon(x^*).
\end{align*}
It follows that
\begin{align*}
    (\lambda A -C) \left(\tr_{\omega(x^*,t^*)} \omega_0(x^*)\right) \Upsilon(x^*) \leq \frac{C}{e^{-t^*}-e^{-T}} \Upsilon(x^*) + C\chi(x^*)^{-4C_2-4}\Upsilon(x^*),
\end{align*}
so by choosing $A$ sufficiently large so that $\lambda A - C = 1$ we have
\begin{equation*}
    (e^{-t^*}-e^{-T})\tr_{\omega(x^*,t^*)} \omega_0(x^*) \leq C,
\end{equation*}
and finally, since $x^* \in U$,
\begin{equation*}
    (\Upsilon Q_1)(x^*,t^*) \leq  C.
\end{equation*}
Therefore we conclude that for all $(x,t) \in X \times [0,T)$,
\begin{equation*}
    (\Upsilon Q_1)(x,t) \leq C
\end{equation*}
so that
\begin{equation*}
    (e^{-t}-e^{-T}) \tr_{\omega(t)} \omega_0 \leq e^{(C\Upsilon^{-1}-A(\d_t \varphi + \varphi))}
\end{equation*}
and the desired result for $\theta = \omega_0$ follows. We extend to the remaining cases via the same argument as the parabolic Schwarz lemma.
\end{proof}

\begin{proposition}
If the (\ref{KRF}) satisfies (\ref{VFC}) then there exists a positive constant $C = C(n,m,\omega_0,f^*\eta,T,K,C_K)$ such that on $K \times [0,T)$ we have
\begin{equation*}
    C^{-1} \omega_{REF}(t) \leq \omega(t) \leq C \omega_{REF}(t).
\end{equation*}
\end{proposition}

\begin{proof}
By the bounds $\tr_{\omega(t)} f^* \eta \leq C$ and $\tr_{\omega(t)} \omega_0 \leq \frac{C}{E(t)}$, we have on $K \times [0,T)$ that
\begin{equation*}
    \omega(t) \geq C^{-1} f^*\eta, \quad \omega(t) \geq C^{-1} E(t) \omega_0.
\end{equation*}
It follows that $0 < \tr_{\omega(t)} \omega_{REF}(t) \leq C$ and
\begin{equation*}
    \omega(t) \geq C^{-1} \omega_{REF}(t).
\end{equation*}
On the other hand, by our assumption (\ref{VFC}), on $K \times [T/2, T)$ we have
\begin{equation*}
    \frac{\omega_{REF}(t)^n}{\omega(t)^n} \geq \frac{\omega_{REF}(t)^n}{C_K E(t)^{n-m} \Omega} \geq \frac{(1-E(t))^m \Omega_{\omega_0,\eta}}{C_K \Omega} \geq C^{-1}.
\end{equation*}
In combination with $\tr_{\omega(t)} \omega_{REF}(t) \leq C$ it follows that
\begin{equation*}
    \omega(t) \leq C \omega_{REF}(t).
\end{equation*}
\end{proof}

\begin{corollary}
If the (\ref{KRF}) satisfies $(\ref{VFC})$ then for any compact set $K' \subseteq B\backslash S'$, there exists a positive constant $C = C(n,m,\omega_0,f^*\eta,T, K', C_{f^{-1}(K')})$ such that for all $(b,t) \in K' \times[0,T)$ we have
\begin{equation*}
     C^{-1} \sqrt{E(t)} \leq \Diam(X_b, \omega(t)|_{X_b}) \leq C \sqrt{E(t)}.
\end{equation*}
Moreover, for any path-connected, compact set $K \subseteq X \backslash S$, there exists a positive constant $C = C(n,m,\omega_0,f^*\eta,T,K,C_K)$ such that for all $t\in [0,T)$ we have
\begin{equation*}
    C^{-1}\Diam(f(K),\eta|_{f(K)}) \leq \Diam(K, \omega(t)|_K) \leq C.
\end{equation*}
\end{corollary}

\begin{proof}
Note that $C^{-1}\omega_{REF}(t) \leq \omega(t) \leq C \omega_{REF}(t)$ implies
\begin{equation*}
    C^{-1} E(t) \omega_{0,b} \leq  \omega(t)|_{X_b} \leq C E(t) \omega_{0,b}.
\end{equation*}
This implies
\begin{equation*}
    \Diam(X_b, C^{-1}E(t)\omega_{0,b}) \leq \Diam(X_b, \omega(t)|_{X_b}) \leq \Diam(X_b, CE(t)\omega_{0,b})
\end{equation*}
and the first result follows from compactness. On the other hand, note that for some $C > 0$, we have $C\omega_0 \geq f^* \eta$ on $X$ and so $C^{-1} \omega_{REF}(t) \leq \omega(t) \leq C \omega_{REF}(t)$ implies
\begin{equation*}
    \omega(t) \leq C \omega_0,
\end{equation*}
on $K \times [0,T)$ and hence
\begin{equation*}
    \Diam(K,\omega(t)|_K) \leq C.
\end{equation*}
Furthermore, on $X \times [T/2,T)$ we have
\begin{equation*}
    \omega_{REF}(t) \geq \frac{1-e^{-T/2}}{1-e^{-T}} f^*\eta
\end{equation*}
so that for any $x,y \in K$ we have
\begin{align*}
d_{\omega_{REF}|_K}(x,y) &= \inf_{\gamma} \int_0^1 \sqrt{g_{REF}(\gamma'(\tau),\gamma'(\tau))} d \tau\\
&\geq {C}^{-1} \inf_{\gamma} \int_0^1 \sqrt{f^*g_{\eta}(\gamma'(\tau),\gamma'(\tau))} d \tau
\end{align*}
where $\gamma:[0,1] \to K$ is any piecewise, smooth curve from $\gamma(0) = x$ to $\gamma(1) = y$. Thus,
\begin{align*}
    d_{\omega_{REF}|_K}(x,y) &\geq C^{-1}\inf_{\gamma} \int_0^1 \sqrt{g_{\eta}((f \circ\gamma)'(\tau),(f\circ \gamma)'(\tau)} d \tau\\
    &\geq C^{-1}\inf_{\phi} \int_0^1 \sqrt{g_{\eta}(\phi'(\tau), \phi'(\tau))} d\tau\\
    &=C^{-1}d_{\eta|_{f(K)}}(f(x),f(y)),
\end{align*}
where $\phi:[0,1] \to f(K)$ is any piecewise, smooth curve from $\phi(0) = f(x)$ to $\phi(1) = f(y)$. It follows that
\begin{equation*}
    \Diam(K,\omega_{REF}(t)|_K) \geq C^{-1} \Diam(f(K),\eta|_{f(K)}).
\end{equation*}
Therefore, as $C^{-1}\omega_{REF}(t) \leq \omega(t)$ on $K \times [0,T)$, we have
\begin{equation*}
    \Diam(K,\omega(t)|_K) \geq C^{-1} \Diam(K,\omega_{REF}(t)|_{K})
\end{equation*}
and the result follows.
\end{proof}

\section{Improved $C^0$-Estimates for the Complex Monge-Ampere Flow}

Now, we recall the construction of the semi-prescribed Ricci curvature form $\omega_{SPR}$, \cite{B25II, ZZ20}. For our purposes, it suffices to know that $\omega_{SPR} = \omega_0 + i\d\dbar \rho_{SPR}$ is a smooth $(1,1)$-form on $X \backslash S$ where $\rho_{SPR} \in C^{\infty}(X \backslash S)$ and the restriction of $\omega_{SPR}$ to any smooth fibre $X_b$, denoted by $\omega_{SPR,b}$, is a K\"ahler metric with Ricci curvature prescribed by
\begin{equation*}
    \Ric \omega_{SPR,b} = \lambda \omega_{0,b}.
\end{equation*}
Furthermore, as in $\cite{B25II, ZZ20}$, there exists a singular, twisted K\"ahler-Einstein metric on $B$, denoted $\omega_{B}$, which is constructed by solving the complex Monge-Ampere equation
\begin{equation}\label{base_space_CMA_1}
    (\eta + i\d\dbar \rho_B)^m = G' e^{\rho_B} \eta^m
\end{equation}
where
\begin{equation*}
    G' = \frac{f_* \Omega}{V \eta^m}
\end{equation*}
for $V = {n\choose m} \int_{X_b} \omega_{0,b}^{n-m}$, and $G = f^*G'$ on $X \backslash S$ where
\begin{equation*}
    G = \frac{\Omega}{{n\choose m} \omega_{SPR}^{n-m} \wedge f^*\eta^m}. 
\end{equation*}
This complex Monge-Ampere equation has a unique solution $\rho_B \in \PSH(B,\eta) \cap C^0(B) \cap C^{\infty}(B \backslash S')$ so that
\begin{equation*}
    \omega_B = \eta + i\d\dbar \rho_B
\end{equation*}
is a singular K\"ahler metric on the base. In fact, $\omega_B$ satisfies the following twisted K\"ahler-Einstein equation
\begin{equation*}
    \Ric \omega_B = -\omega_B - \lambda \eta + \omega_{WP,\lambda}
\end{equation*}
where $\omega_{WP,\lambda}$ is a $(1,1)$-form related to the Weil-Petersson metric constructed in \cite{B25II}.\\
Additionally, the complex Monge-Ampere equation
\begin{equation*}
    \left(\frac{1}{1-e^{-T}}\eta + i\d\dbar \rho_B'\right)^m = G' e^{\rho_B'} \left(\frac{1}{1-e{^-T}}\eta\right)^m
\end{equation*}
also has a unique solution $\rho_B' \in \PSH(B,\frac{1}{1-e^{-T}}\eta) \cap C^0(B) \cap C^{\infty}(B \backslash S')$ so that
\begin{equation*}
    \omega_B' = \frac{1}{1-e^{-T}} \eta + i\d\dbar \rho_B'
\end{equation*}
is another singular K\"ahler metric on $B$. The metric $\omega_B'$ satisfies the following twisted K\"ahler-Einstein equation on $ B \backslash S'$, \cite{B25II},
\begin{equation*}
    \Ric \omega_B' = -\omega_B' + \omega_{WP,\lambda}.
\end{equation*}

\noindent Furthermore, it will be of interest to consider the case where $f: X \to B$ generically admits a smoothly varying family of K\"ahler-Einstein metrics. To be precise, we assume there exists a $p > 1$ and $\rho_{SKE} \in C^{\infty}(X\backslash S)$ such that $e^{-\lambda \rho_{SKE}} \in L^p(X,\Omega)$ and $\rho_{SKE}$ can have a logarithmic pole of a small order from above along $S$. Additionally, the semi K\"ahler-Einstein form $\omega_{SKE} = \omega_0 + i\d\dbar \rho_{SKE}$ is a $d$-closed, real $(1,1)$-form on $X \backslash S$ for which we require the restriction $\omega_{SKE,b} := \omega_{SKE}|_{X_b}$ is a K\"ahler metric satisfying
\begin{equation}\tag{SKE}\label{SKE}
    \Ric \omega_{SKE,b} = \lambda \omega_{SKE,b}.
\end{equation}
As in \cite{B25II}, given the assumption (\ref{SKE}), one can again construct suitable twisted K\"ahler-Einstein metrics $\omega_B$, $\omega_B'$, on the base by replacing $G$ and $G'$ above by
\begin{equation*}
    G' = \frac{f_*( e^{-\lambda \rho_{SKE}}\Omega)}{V \eta^m}
\end{equation*}
where $V = {n \choose m} \int_{X_b} \omega_{0,b}^{n-m}$ and
\begin{equation*}
    G = \frac{e^{-\lambda \rho_{SKE}}\Omega}{{n\choose m} \omega_{SKE}^{n-m} \wedge f^*\eta^m},
\end{equation*}
and adapting the definition of the $(1,1)$-form $\omega_{WP,\lambda}$.

\begin{proposition}\label{potential_convergence}
If (\ref{KRF}) satisfies (\ref{VFC}) then for any compact set $K \subseteq X \backslash S'$ there exists a $C = C(n,m,T,\omega_0,f^*\eta, \omega_{SPR},K, C_{K})> 0$ such that on $K \times [0,T)$ we have
\begin{equation*}
    |\varphi(t) - f^*\varphi_{avg}(t)| \leq CE(t),
\end{equation*}
where $\varphi_{avg}$ denotes the fibre-wise average of $\varphi$ with respect to $\omega_{0,b}$.
\end{proposition}

\begin{proof}
Let $\varphi_{avg}$ be a smooth function on $B\backslash S' \times [0,T)$ defined by
\begin{equation*}
    \varphi_{avg} = \frac{1}{\Vol(X_b,\omega_{0,b})} \int_{X_b} \varphi \omega_{0,b}^{n-m}.
\end{equation*}
Set $Q_2 = \frac{\varphi - f^* \varphi_{avg}}{E(t)}$, and note that 
\begin{equation*}
    E(t)^{-1}\omega(t)|_{X_b} = \omega_{0,b} + i\d\dbar Q_2|_{X_b}
\end{equation*}
satisfies the following equation
\begin{equation*}
    (\omega_{0,b} + i\d\dbar Q_2|_{X_b})^{n-m} = E(t)^{-(n-m)} \omega(t)|_{X_b}^{n-m} = G_b \omega_{0,b}^{n-m}
\end{equation*}
for some function $G_b > 0$. Since, $\int_{X_b} Q_2 \omega_{0,b}^{n-m}= 0$, we can apply Yau's $L^\infty$-estimate to get
\begin{equation*}
    \sup_{X_b \times [0,T)} |Q_2| \leq C_b,
\end{equation*}
where $C_b = C_b(n,m, T, C_{Sob}(\omega_{0,b}), C_{Poin}(\omega_{0,b}), \sup_{X_b \times [0,T)} G_b)$, \cite{Y78}, \cite[Thm. 3]{PSS12}, and $C_{Poin}(\omega_{0,b})$ and $C_{Sob}(\omega_{0,b})$ denote the Poincare and Sobolev constants of $(X_b, \omega_{0,b})$. The result will follow from a uniform bound on the constants $\sup_{X_b \times [0,T)} G_b$, $C_{Poin}(\omega_{0,b})$, and $C_{Sob}(\omega_{0,b})$ for all $b \in K'$. Firstly, note that on $K \times [0,T)$,
\begin{align*}
    \frac{\omega(t)|_{X_b}^{n-m}}{\omega_{0,b}^{n-m}} &= \frac{\omega(t)^{n-m} \wedge f^* \eta^m}{\omega_0^{n-m} \wedge f^* \eta ^m}\\
    &=\frac{\omega(t)^{n-m} \wedge f^* \eta^m}{\omega(t)^n} \frac{\omega(t)^n}{\omega_0^{n-m} \wedge f^*\eta^m}\\
    &\leq C (\tr_{\omega(t)} f^* \eta)^m \frac{\omega(t)^n}{\Omega}\\
    &\leq C E(t)^{n-m}
\end{align*}
where we use the AM-GM inequality in the second last step, and the last step holds by (\ref{VFC}). So, $G_b \leq C = C(n,m,T,\omega_0, f^* \eta, C_{K})$. On the other hand, the argument for the uniform bounds on the Sobolev and Poincare constants is similar to that in \cite{T10, ZZ20, ST07}, so we will highlight the main points.\\
First, we assume that the fibre is not $1$-dimensional, i.e. $n-m > 1$. The Sobolev constants are bounded independently of $b \in B \backslash S'$ in \cite[Lem. 3.2]{T10}. That is, there exists a uniform constant $C > 0$ for all $b \in B \backslash S'$ such that
\begin{equation*}
    \left(\int_{X_b} |v|^{\frac{2(n-m)}{n-m-1}} \omega_{0,b}^{n-m} \right)^{\frac{n-m-1}{n-m}} \leq C \int_{X_b} (|\nabla v|_{\omega_{0,b}}^2 + |v|^2) \omega_{0,b}^{n-m}
\end{equation*}
for all $v \in C^\infty(X_b)$.\\
On the other hand, the Poincare constants can be bounded independently of $b \in K'$ by the argument of \cite[Lem. 3.4]{T10},(cf.\cite{T07},\cite{LY86}). That is, there exists a uniform constant $C = C(K') > 0 $ such that for all $b \in K'$ we have
\begin{equation*}
    \int_{X_b} |v|^2 \omega_{0,b}^{n-m} \leq C \int_{X_b} |\nabla v|_{\omega_{0,b}}^2 \omega_{0,b}^{n-m}
\end{equation*}
for all $v \in C^\infty(X_b)$ with $\int_{X_b} v \omega_{0,b}^{n-m} = 0$.\\
Now as in \cite{ZZ20, ST07}, if the fibre is $1$-dimensional, i.e. $n-m =1$, if we let $G_{\omega_{SPR,b}}$ be the Green's function with respect to $\omega_{SPR,b}$ and set $\underline{G}_b = \inf_{X_b \times X_b} G_{\omega_{SPR,b}}(\cdot,\cdot)$, then for any $x \in X_b$ Green's formula will yield
\begin{align*}
    \left|\varphi(x) - \frac{1}{\Vol(\omega_{SPR,b})}\int_{X_b} \varphi(y) \omega_{SPR,b}(y)\right| &= \left|\int_{X_b} (\Delta_{\omega_{SPR,b}} \varphi)(y) (G_{\omega_{SPR,b}}(x,y) - \underline{G}_b)\omega_{SPR,b}(y)\right|\\
    &\leq \sup_{X_b} |\Delta_{\omega_{SPR,b}} \varphi| \int_{X_b} (G_{\omega_{SPR,b}}(x,y) - \underline{G}_b) \omega_{SPR,b}.
\end{align*}
We note that
\begin{equation*}
    |\Delta_{\omega_{SPR,b}} \varphi| = |\Delta_{\omega_{0,b}} \varphi| \frac{\omega_{0,b}}{\omega_{SPR,b}}
\end{equation*}
and both terms can be bounded as follows. Firstly,
\begin{equation*}
    0 < E(t) + \Delta_{\omega_{0,b}} \varphi = \tr_{\omega_{0,b}} (\omega(t)|_{X_b}) = \frac{\omega(t)|_{X_b}}{\omega_{0,b}} \leq C \frac{\omega_{REF}(t)|_{X_b}}{\omega_{0,b}} \leq CE(t),
\end{equation*}
so that $|\Delta_{\omega_{0,b}} \varphi| \leq CE(t)$ for a uniform constant $C>0$ for all $b \in K'$, and 
\begin{equation*}
    \frac{\omega_{0,b}}{\omega_{SPR,b}} = \frac{\omega_0 \wedge f^*\eta^{n-1}}{\omega_{SPR} \wedge f^*\eta^{n-1}} \leq C
\end{equation*}
on $K$ for some $C = C(\omega_0, \omega_{SPR}, f^*\eta, K)$. Furthermore, the Green's functions $G_{\omega_{SPR,b}}$ are uniformly bounded above for all $b \in K'$, while we can also bound $\underline{G}_b$ below uniformly for all $b \in K'$ using the same argument as \cite[Lem. 3.4]{ZZ20} by
\begin{equation*}
    \underline{G}_b \geq -C \frac{\Diam(X_b,\omega_{SPR,b})^2}{\Vol(X_b,\omega_{SPR,b})} \geq -C
\end{equation*}
for some uniform constant $C > 0$ for all $b \in K'$. Therefore, there exists a uniform $C > 0$ for all $b \in K'$ such that
\begin{equation*}
    \left|\varphi - \frac{1}{\Vol(\omega_{SPR,b})} \int_{X_b} \varphi \omega_{SPR,b}\right| \leq CE(t).
\end{equation*}
It follows that 
\begin{align*}
    \left|\varphi - \frac{1}{\Vol(\omega_{0,b})} \int_{X_b} \varphi \omega_{0,b}\right|
    &\leq \sup_{X_b} \varphi - \inf_{X_b} \varphi\\
    &= \sup_{X_b} \left(\varphi - \frac{1}{\Vol(\omega_{SPR,b})}\int_{X_b} \varphi \omega_{SPR,b}\right)\\
    &- \inf_{X_b} \left(\varphi - \frac{1}{\Vol(\omega_{SPR,b})}\int_{X_b} \varphi \omega_{SPR,b} \right) \leq CE(t).
\end{align*}
\end{proof}

\begin{proposition}
If (\ref{KRF}) satisfies (\ref{VFC}) then for any compact set $K \subseteq X \backslash S$ there exists a $C = C(n,m,T,\omega_0,f^*\eta, \omega_{SPR}, K, C_K) > 0$ such that on $K \times [0,T)$ we have
\begin{equation*}
    |u(t)-f^*u_{avg}(t)| \leq CE(t),
\end{equation*}
where $u_{avg}$ is the fibre-wise average of $u$ with respect to $\omega_{0,b}$.
\end{proposition}

\begin{proof}
We apply $\left|\frac{\d\varphi}{\d t}\right| \leq C$ and Proposition \ref{potential_convergence} to the following triangle inequality
\begin{equation*}
    |u-f^* u_{avg}| \leq |u-\varphi| + |\varphi - f^* \varphi_{avg}| + |f^* \varphi_{avg} - f^* u_{avg}| \leq CE(t).
\end{equation*}
\end{proof}

\noindent We now come to another main theorem.

\begin{theorem}\label{C0_varphi_estimate}
Assume the (\ref{KRF}) satisfies (\ref{VFC}). For any compact set $K \subseteq X \backslash S$ there exists a real function $h_K = h_K(n,m,T,\omega_0,f^*\eta,\rho_{SPR},f^*\rho_B,U_X)$, where $h_K:[0,T) \to [0,\infty)$ is continuous, strictly decreasing and satisfies $h_K \to 0$ as $t \to T$ such that, on $K \times [0,T)$, we have
\begin{equation*}
    \lambda \inf_{X \backslash S} (\rho_{SPR} - f^*\rho_B) - h_K(t) \leq \varphi(t) - f^*\rho_B \leq h_K(t) + \lambda \sup_{X \backslash S} (\rho_{SPR} - f^*\rho_B).
\end{equation*}
Moreover, there exists a $h_K:[0,T) \to [0,\infty)$, where $h_K = h_K(n,m,T,\omega_0,f^*\eta,\rho_{SPR},f^*\rho'_B,U_X)$, such that on $K \times [0,T)$
\begin{equation*}
    \lambda \inf_{X \backslash S} \rho_{SPR} - h_K(t) \leq \varphi(t) - (1-e^{-T}) f^*\rho_B' \leq h_K(t) + \lambda \sup_{X \backslash S} \rho_{SPR}.
\end{equation*}
Furthermore, if (\ref{SKE}) holds then there exists a $h_K: [0,T) \to [0,\infty)$, where $h_K = h_K(n,m,T,\omega_0,f^*\eta,\rho_{SKE},f^*\rho_B, U_X)$, such that on $K \times [0,T)$
\begin{equation*}
    \lambda \inf_{X \backslash S} (- f^*\rho_B) - h_K(t) \leq \varphi(t) - f^*\rho_B \leq h_K(t) + \lambda \sup_{X \backslash S} (- f^*\rho_B).
\end{equation*}
Finally, if (\ref{SKE}) holds then there exists a $h_K: [0,T) \to [0,\infty)$, where $h_K = h_K(n,m,T,\omega_0,f^*\eta,\rho_{SKE},f^*\rho'_B, U_X)$, such that on $K \times [0,T)$
\begin{equation*}
     - h_K(t) \leq \varphi(t) - (1-e^{-T})f^*\rho'_B \leq h_K(t).
\end{equation*}
\end{theorem}

\begin{proof}
We can extend the argument \cite[Thm 5.1, Prop 5.4]{ST12} of Song and Tian to this setting. Let $L'$ be an ample $\Q$-line bundle on $B$ such that $D = f^*L'$ and choose an ample divisor $D'$ on $B$ such that $[D'] = \mu'L'$ for some sufficiently large integer $\mu'$, and we have $S' \subseteq D'$. In fact, we can choose finitely many such divisors $D_j$ for $j = 1, \dots, M$, such that $\bigcap_{j = 1}^M D_j = S'$. Fix a divisor $D_j$, let $s_j$ denote a defining section of $D_j$ and set $h_{D_j} = (h_{FS})^{\mu_j/k} e^{-\mu_j\rho_B}$, where $h_{FS}$ is the Fubini Study metric on $\OO_B(1)$, to be a continuous Hermitian metric on $D_j$ so that
\begin{equation*}
    \Ric h_{D_j} = -i\d\dbar \log |s_j|^2_{h_{D_j}} = \mu_j\omega_B.
\end{equation*}
Set $B_{\eta}(D_j,r)$ to be the tubular neighbourhood around $D_j$ of radius $r > 0$ with respect to the metric $\eta$. As $s_j$ is the defining section of $D_j$ then $s_j$ vanishes along $D_j$ and $\log|s_j|_{h_{D_j}}^2 = -\infty$ on $D_j$, moreover, we can scale $s_j$ such that $\log |s_j|_{h_{D_j}}^2 \leq -1$ globally on $X$, where we identify $\log |s_j|_{h_{D_j}}^2$ with its pullback. As $\varphi$ and $f^*\rho_B$ are bounded globally on $X$, for any $0 < \e < 1$ we can choose a radius $r_\e > 0$ so that $r_\e \to 0$ as $\e \to 0$ and on $f^{-1}(B_{\eta}(D_j,r_\e))$ we have both
\begin{align}\label{tubular_nbhd_bound}
    \e \log |s_j|_{h_{D_j}}^2 &\leq -C_1 \leq -1 - (\varphi - (1-E(t))f^* \rho_B),\\
    \e \log |s_j|_{h_{D_j}}^2 &\leq -C_2 \leq -1 - ((1-E(t))f^* \rho_B - \varphi).\notag
\end{align}
Let $\chi_\e$ be a smooth cut-off function such that $0 \leq \chi_\e \leq 1$ on $B$ and
\begin{equation*}
\chi_\e = 1 \text{ on } B \backslash B_{\eta}(D_j,r_\e), \quad \text{and} \quad \chi_\e = 0 \text{ on } B_{\eta}(D_j,r_\e/2)
\end{equation*}
and let us use the notation
\begin{align*}
    \rho_{SPR,\e} = (f^*\chi_{\e}) \rho_{SPR}, \quad
    \omega_{SPR,\e} = \omega_0 + i\d\dbar \rho_{SPR,\e}.
\end{align*}
Next, we define two auxiliary functions $\varphi^{+}_{\e}, \varphi^{-}_{\e}$ for $\varphi$ by
\begin{align*}
    \varphi_{\e}^{+}(t) &= \varphi(t) - \left(\frac{e^{-t}-e^{-T}}{1-e^{-T}}\right) \rho_{SPR,\e} - \left(\frac{1-e^{-t}}{1-e^{-T}}\right) f^*\rho_{B} + \e \log|s_j|_{h_{D_j}}^2,\\
    \varphi_{\e}^{-}(t) &= \varphi(t) - \left(\frac{e^{-t}-e^{-T}}{1-e^{-T}}\right) \rho_{SPR,\e} - \left(\frac{1-e^{-t}}{1-e^{-T}}\right) f^*\rho_{B} - \e \log|s_j|_{h_{D_j}}^2.
\end{align*}
As a result, we can write
\begin{align*}
    \omega(t) &= \left(\frac{e^{-t}-e^{-T}}{1-e^{-T}}\right) \omega_{SPR, \e} + \left(\frac{1-e^{-t}}{1-e^{-T}} + \mu_j\e\right) f^* \omega_B+ i\d\dbar\varphi_{\e}^+(t),
\end{align*}
and
\begin{align*}
    \omega(t) &= \left(\frac{e^{-t}-e^{-T}}{1-e^{-T}}\right) \omega_{SPR, \e} + \left(\frac{1-e^{-t}}{1-e^{-T}} - \mu_j\e\right) f^* \omega_B+ i\d\dbar\varphi_{\e}^{-}(t).
\end{align*}
And furthermore, it follows from (\ref{base_space_CMA_1}) that
\begin{equation*}
     {n \choose m} \omega_{SPR}^{n-m} \wedge f^* \omega_B^m =  {n \choose m} \omega_{SPR}^{n-m} \wedge (G e^{f^*\rho_B} f^* \eta^{m}) = e^{f^* \rho_B} \Omega.
\end{equation*}
Therefore, we find that
\begin{align*}
    \frac{\d \varphi_{\e}^+}{\d t} = \log \frac{E(t)^{-(n-m)} \omega(t)^n}{{n \choose m} \omega_{SPR}^{n-m} \wedge f^* \omega_B^m} - \varphi_{\e}^{+} + \lambda \rho_{SPR,\e} - \lambda f^* \rho_B + \e \log |s_j|^2_{h_{D_j}},
\end{align*}
and
\begin{equation*}
    \frac{\d \varphi_{\e}^-}{\d t} = \log \frac{E(t)^{-(n-m)} \omega(t)^n}{{n \choose m} \omega_{SPR}^{n-m} \wedge f^* \omega_B^m} - \varphi_{\e}^{-} + \lambda \rho_{SPR,\e} - \lambda f^* \rho_B - \e \log |s_j|^2_{h_{D_j}}.
\end{equation*}
Choose $T^+_{\e}$ sufficiently large such that for all $t \geq T^+_{\e}$
\begin{align*}
    A_{\e}^{+}(t)&:= \sum_{i=0}^{m-1} E(t)^{m-i}(1-E(t)+\mu_j\e)^i \frac{{n\choose i}\omega_{SPR}^{n-i} \wedge f^* \omega_B^i}{{n\choose m} \omega_{SPR}^{n-m} \wedge f^*\omega_B^m}
\end{align*}
satisfies
\begin{equation*}
    A_{\e}^+(t) \leq \e
\end{equation*}
on $X \backslash f^{-1}(B_{\eta}(D_j,r_\e))$. Additionally, choose $T^-_{\e}$ sufficiently large such that for all $t \geq T^-_{\e}$
\begin{align*}
    A_{\e}^{-}(t) &:= \sum_{i=0}^{m-1} E(t)^{m-i}(1-E(t)-\mu_j\e)^i \frac{{n\choose i}\omega_{SPR}^{n-i} \wedge f^* \omega_B^i}{{n\choose m} \omega_{SPR}^{n-m} \wedge f^*\omega_B^m}
\end{align*}
satisfies
\begin{equation*}
    - \e \leq A^-_{\e}(t)
\end{equation*}
on $X \backslash f^{-1}(B_{\eta}(D_j,r_{\e}))$.\\
Moreover, we can choose $T'_{\e}$ sufficiently large such that for all $t \geq T'_{\e}$,
\begin{equation*}
    E(t) \sup_{f^{-1}(B_{\eta}(D_j,r_\e))}|\rho_{SPR,\e}| \leq 1,
\end{equation*}
and $T''_{\e}$ sufficiently large such that for all $t \geq T''_{\e}$ we have $E(t) \leq \e$.\\
Set $T_{\e} = \max\{T_{\e}^+, T_{\e}^{-}, T'_{\e}, T''_{\e}\}$. Let $(x^*,t^*)$ be a maximal point of $\varphi_{\e}^{+}$ on $X \times [T_\e,T)$ and assume that $x^* \in X \backslash f^{-1}(B_{\eta}(D_j, r_\e))$. By the maximum principle, and as $\omega_{SPR,\e}(x^*) = \omega_{SPR}(x^*)$, it follows that at $(x^*,t^*)$ we have
\begin{align*}
    \varphi_{\e}^{+}(x^*,t^*) &\leq \log \left((1-E(t^*)+\mu_j \e)^m + A^+_{\e}(x^*,t^*)\right)\\
    &+ \lambda \rho_{SPR}(x^*) - \lambda f^* \rho_B(x^*) + \e \log |s_j|^2_{h_{D_j}}(x^*)\\
    &\leq C \e + \lambda \sup_{X \backslash f^{-1}(B_{\eta}(D_j, r_\e))} (\rho_{SPR}-f^*\rho_B),
\end{align*}
where $C = C(m,\mu_j)$ does not depend on $\e$. On the other hand, if $x^* \in f^{-1}(B_{\eta}(D_j,r_\e))$ then by (\ref{tubular_nbhd_bound}) we have
\begin{align*}
    \varphi^+_{\e}(x^*,t^*) &= \varphi(x^*,t^*) - E(t^*) \rho_{SPR,\e}(x^*) - (1-E(t^*)) f^*\rho_{B}(x^*) + \e \log |s_j|^2_{h_{D_j}}(x^*)\\
    &\leq -1 + E(t^*) \sup_{f^{-1}(B_{\eta}(D_j,r_\e))} |\rho_{SPR,\e}|\\
    &\leq 0.
\end{align*}
We note that we can normalize $\rho_{SPR}$ or $\rho_B$ such that the supremum and infimum of $\rho_{SPR} - f^*\rho_B$ are positive and negative on $X \backslash S$, respectively. Thus
\begin{equation*}
    \varphi^+_{\e}(t) \leq C\e + \lambda \sup_{X\backslash S} (\rho_{SPR}-f^*\rho_B)
\end{equation*}
holds globally on $X \times [T_\e,T)$.\\
Similarly, let $(x_*,t_*)$ be a point at which $\varphi_{\e}^-(t)$ achieves a minimum on $X \times [T_\e,T)$ and take $x_* \in X \backslash f^{-1}(B_{\eta}(D_i,r_\e))$. It follows, as $\omega_{SPR,\e}(x_*) = \omega_{SPR}(x_*)$, that 
\begin{align*}
    \varphi^{-}_{\e}(x_*,t_*) &\geq \log\left((1-E(t_*)-\mu_j \e)^m + A_{\e}^{-}(x_*,t_*)\right)\\
    &+ \lambda \rho_{SPR}(x_*) - \lambda f^*\rho_{B}(x_*) - \e \log |s_j|^2_{h_{D_j}}(x_*) \\
    &\geq \log \left((1-E(t_*) - \mu_j \e)^m - \e\right)\\
    &+ \lambda \inf_{X\backslash f^{-1}(B_{\eta}(D_j,r_\e))} (\rho_{SPR} - f^*\rho_B).
\end{align*}
Observing that since $t_* \geq T_\e \geq T''_\e$ we have $E(t_*) \leq \e$, there exists a small $\e' > 0$ such that for all $0 < \e \leq \e'$ 
\begin{equation*}
    \varphi^{-}_{\e}(x_*,t_*) \geq -C \e + \lambda \inf_{X \backslash f^{-1}(B_{\eta}(D_j,r_{\e}))
    } (\rho_{SPR} - f^*\rho_B),
\end{equation*}
for some $C = C(m,\mu_j, \e')$ not dependent on $\e$. On the other hand, if $x_* \in f^{-1}(B_{\eta}(D_i,r_{\e}))$ then by (\ref{tubular_nbhd_bound})
\begin{align*}
    \varphi_{\e}^{-}(x_*,t_*) &= \varphi(x_*,t_*) - E(t_*) \rho_{SPR,\e}(x_*,t_*) - (1-E(t_*)) f^*\rho_{B}(x_*) - \e \log|s_j|_{h_{D_j}}^2(x_*)\\
    &\geq 1-E(t_*) \sup_{f^{-1}(B_{\eta}(D_j,r_\e))} |\rho_{SPR,\e}|\\
    &\geq 0.
\end{align*}
So on $X \times [T_\e,T)$ we have
\begin{equation*}
    \varphi_{\e}^{-}(t) \geq -C\e + \lambda  \inf_{X\backslash S} (\rho_{SPR} - f^*\rho_B).
\end{equation*}
Next, for any compact set $K \subseteq X \backslash S$ we can choose $\e_0$ sufficiently small such that $K \subseteq X \backslash f^{-1}(B_{\eta}(S',r_\e))$ for all $0 < \e \leq \e_0$ and there are finitely many compact sets $K_j$ such that  $K \subseteq \bigcup_{j=1}^M K_j$ and $K_j \subseteq X \backslash f^{-1}(B_{\eta}(D_j, r_{\e}))$ for all $0 < \e \leq \e_0$. Therefore, on $K_j$ and for all $0 < \e \leq \e_0$,  we have
\begin{align*}
    \varphi(T_\e) - f^*\rho_B &= \varphi_{\e}^{+}(T_\e) + \frac{e^{-T_\e}-e^{-T}}{1-e^{-T}}\rho_{SPR,\e} - \frac{e^{-T_\e}-e^{-T}}{1-e^{-T}} f^*\rho_B - \e \log |s_i|_{h_{D_i}}^2\\
    &\leq C_j \e + \lambda \sup_{X \backslash S} (\rho_{SPR}-f^*\rho_B)+ C_j\frac{e^{-T_\e}-e^{-T}}{1-e^{-T}}.
\end{align*}
Now note that, by construction, $T_\e = \max\{T^+_{\e},T^-_{\e},T'_{\e},T''_{\e}\}$ is a strictly increasing, continuous function of $\e$ that approaches $T$ as $\e$ approaches $0$. This means there exists a strictly decreasing, continuous, bijective function $h: [T_{\e_0}, T) \to (0,\e_0]$ such that $h(T_{\e}) = \e$ and clearly $h(t) \to 0$ as $t \to T$. So, we can take the maximum over all $j$ in the previous inequality to define a strictly decreasing, continuous function $h_K$ depending on the compact set $K$, where $h_K(t) \to 0$ as $t \to T$, such that 
\begin{equation*}
    \varphi(t) - f^*\rho_B \leq h_K(t) + \lambda \sup_{X \backslash S} (\rho_{SPR}-f^*\rho_B)
\end{equation*}
holds on $K \times [T_{\e_0},T)$, and the lower bound is similar.\\
This completes the main portion of the proof. We now comment on the minor adjustments needed for the other cases in the theorem. For the $C^0$-estimate involving $\varphi(t)-(1-e^{-T})f^*\rho_B'$, one first needs to set $h_{D_j} = (h_{FS})^{\mu_j/k} e^{-\mu_j(1-e^{-T})\rho_B'}$ so that
\begin{equation*}
    \Ric h_{D_j} = -i\d\dbar \log |s_j|^2_{h_{D_j}} = \mu_j(1-e^{-T})\omega_B',
\end{equation*}
and choose the radius $r_\e > 0$ such that both
\begin{align*}
    \e \log |s_j|_{h_{D_j}}^2 &\leq -C_1 \leq -1 - (\varphi - (1-E(t)) (1-e^{-T}) f^*\rho_B'),\\
    \e \log |s_j|_{h_{D_j}}^2 &\leq -C_2 \leq -1 - ((1-E(t)) (1-e^{-T}) f^*\rho_B' - \varphi)
\end{align*}
hold on $f^{-1}(B_{\eta}(D_j,r_\e))$. Additionally, one instead takes the auxiliary functions
\begin{align*}
    \varphi_{\e}^{+}(t) &= \varphi(t) - \left(\frac{e^{-t}-e^{-T}}{1-e^{-T}}\right) \rho_{SPR,\e} - \left(\frac{1-e^{-t}}{1-e^{-T}}\right) (1-e^{-T})f^*\rho'_{B} + \e \log|s_j|_{h_{D_j}}^2,\\
    \varphi_{\e}^{-}(t) &= \varphi(t) - \left(\frac{e^{-t}-e^{-T}}{1-e^{-T}}\right) \rho_{SPR,\e} - \left(\frac{1-e^{-t}}{1-e^{-T}}\right) (1-e^{-T})f^*\rho'_{B} - \e \log|s_j|_{h_{D_j}}^2.
\end{align*}
Their evolution equations can be calculated to be
\begin{align*}
    \frac{\d \varphi_{\e}^+}{\d t} = \log \frac{E(t)^{-(n-m)} \omega(t)^n}{{n \choose m} \omega_{SPR}^{n-m} \wedge ((1-e^{-T})f^* \omega'_B)^m} - \varphi_{\e}^{+} + \lambda \rho_{SPR,\e} + \e \log |s_j|^2_{h_{D_j}},
\end{align*}
and
\begin{equation*}
    \frac{\d \varphi_{\e}^-}{\d t} = \log \frac{E(t)^{-(n-m)} \omega(t)^n}{{n \choose m} \omega_{SPR}^{n-m} \wedge ((1-e^{-T})f^* \omega'_B)^m} - \varphi_{\e}^{-} + \lambda \rho_{SPR,\e} - \e \log |s_j|^2_{h_{D_j}},
\end{equation*}
since
\begin{equation*}
    {n \choose m} \omega_{SPR}^{n-m} \wedge ((1-e^{-T})f^*\omega'_B)^m = e^{f^*\rho_B'} \Omega.
\end{equation*}
The remaining arguments will follow as before.\\
Finally, consider the second case involving $\rho_{SKE}$. We note that the volume form $\Omega$ satisfies
\begin{equation*}
    {n\choose m} \omega_{SKE}^{n-m} \wedge ((1-e^{-T}) f^*\omega'_B)^m = e^{f^*\rho_B'-\lambda \rho_{SKE}}\Omega,
\end{equation*}
and the auxiliary functions
\begin{align*}
    \varphi_{\e}^{+}(t) &= \varphi(t) - \left(\frac{e^{-t}-e^{-T}}{1-e^{-T}}\right) \rho_{SKE,\e} - \left(\frac{1-e^{-t}}{1-e^{-T}} \right) (1-e^{-T})f^*\rho'_{B} + \e \log|s_j|_{h_{D_j}}^2,\\
    \varphi_{\e}^{-}(t) &= \varphi(t) - \left(\frac{e^{-t}-e^{-T}}{1-e^{-T}}\right) \rho_{SKE,\e} - \left(\frac{1-e^{-t}}{1-e^{-T}}\right) (1-e^{-T})f^*\rho'_{B} - \e \log|s_j|_{h_{D_j}}^2.
\end{align*}
have evolution equations
\begin{align*}
    \frac{\d \varphi_{\e}^+}{\d t} &= \log \frac{E(t)^{-(n-m)} \omega(t)^n}{{n \choose m} \omega_{SKE}^{n-m} \wedge ((1-e^{-T})f^* \omega'_B)^m} - \varphi_{\e}^{+} + \e \log |s_j|^2_{h_{D_j}},\\
    \frac{\d \varphi_{\e}^-}{\d t} &= \log \frac{E(t)^{-(n-m)} \omega(t)^n}{{n \choose m} \omega_{SKE}^{n-m} \wedge ((1-e^{-T})f^* \omega'_B)^m} - \varphi_{\e}^{-} - \e \log |s_j|^2_{h_{D_j}}.
\end{align*}
As such we will get convergence.
\end{proof}

\noindent It is immediate that,

\begin{corollary}
Assume the (\ref{KRF}) satisfies (\ref{VFC}). If $\rho_{SPR} - f^*\rho_B = 0$ on $X \backslash S$, then 
\begin{equation*}
\norm{\varphi(t) - f^*\rho_B}_{C^0_{loc}(X\backslash S)} \to 0.
\end{equation*}
If $\rho_{SPR} = 0$ on $X \backslash S$ then 
\begin{equation*}
\norm{\varphi(t) - (1-e^{-T}) f^*\rho'_B}_{C^0_{loc}(X \backslash S)} \to 0.
\end{equation*}
Suppose that (\ref{SKE}) holds. If $f^* \rho_B = 0$ on $X \backslash S$ then 
\begin{equation*}
\norm{\varphi(t) - f^*\rho_B}_{C^0_{loc}(X \backslash S)} \to 0.
\end{equation*}
Suppose that (\ref{SKE}) holds. Then 
\begin{equation*}
\norm{\varphi(t) - (1-e^{-T})f^*\rho'_B}_{C^0_{loc}(X \backslash S)} \to 0.
\end{equation*}
\end{corollary}

\noindent Now, we would like more control on the function $h_K$ in the proof of Theorem \ref{C0_varphi_estimate}.

\begin{lemma}\label{lipschitz_lemma}
If the function $h(T_\e) = \e$ is Lipschitz then there exists a constant $C = C(n,m,T,\omega_0,f^*\eta,\rho_{SPR},f^*\rho_B,U_X) > 0$ such that on $[0,T)$
\begin{equation*}
h_K(t) \leq C E(t).
\end{equation*}
Moreover, the dependence of the constant $C$ on $\rho_{SPR}$ and $f^*\rho_B$ should be appropriately replaced with $\rho_{SKE}$ or $f^*\rho_B'$ if necessary.
\end{lemma}

\begin{proof}
In order to prove $h_K(t) \leq CE(t)$ we simply need to prove $h:[T_{\e_0}, T) \to (0,\e_0]$ is bounded by $CE(t)$ but this is immediate, by definition. Indeed, if $h$ is Lipschitz then there exists a $C > 0$ such that for all $t,\tau\in [T_{\e_0},T)$
\begin{equation*}
    |h(\tau)-h(t)| \leq C|\tau-t|.
\end{equation*}
Taking the limit as $\tau \to T$ yields the desired result since $(T-t) \sim E(t)$.
\end{proof}

\noindent When $f:X \to B$ is a submersion we have the following improvement on Theorem \ref{C0_varphi_estimate}.

\begin{theorem}
Assume that the map $f: X \to B$ is a submersion. If $\omega_{SPR} - \omega_0 + f^*\omega_B - f^*\eta = 0$ on $X$, then there exists a constant $C = C(n,m,T,\omega_0,f^*\eta,\rho_{SPR},f^*\rho_B) > 0$ such that on $X \times [0,T)$ we have
\begin{equation*}
    \left|\varphi(t) - f^*\rho_B\right| \leq CE(t).
\end{equation*}
If $\omega_{SPR} = \omega_{0}$ on $X$, then there exists a constant $C = C(n,m,T,\omega_0,f^*\eta, \rho_{SPR}, f^*\rho_B') > 0$ such that on $X \times [0,T)$ we have
\begin{equation*}
    \left|\varphi(t) - (1-e^{-T}) f^*\rho_B'\right| \leq CE(t).
\end{equation*}
If (\ref{SKE}) holds and $f^*\omega_{B} = f^*\eta$ on $X$, then there exists a positive constant $C = C(n,m,T,\omega_0,f^*\eta,\rho_{SKE}, f^*\rho_B) > 0$ such that on $X \times [0,T)$ we have
\begin{equation*}
    \left|\varphi(t) - f^*\rho_B\right| \leq CE(t).
\end{equation*}
If (\ref{SKE}) holds then there exists a constant $C = C(n,m,T,\omega_0,f^*\eta, \rho_{SKE},f^*\rho_B') > 0$ such that on $X \times [0,T)$ we have
\begin{equation*}
    \left|\varphi(t) - (1-e^{-T}) f^*\rho_B'\right| \leq CE(t).
\end{equation*}
\end{theorem}

\begin{proof}
First, we note that the only dependence on $U_X$ in the proof of Theorem \ref{C0_varphi_estimate} is in the global $C^0$-bound for $\varphi$. However, due to Proposition \ref{submersion_C0_varphi_estimate}, we have a $C^0$-bound for $\varphi$ independent of $U_X$ so we can drop the (\ref{VFC}) assumption.\\
Let us simply treat the first case listed since the argument for the other cases is identical. If $\omega_{SPR} - \omega_0 + f^*\omega_B - f^*\eta = 0$ on $X$ then $\rho_{SPR}-f^*\rho_B$ is a constant which can be normalized to be $0$. As such, it follows that
\begin{equation*}
    |\varphi(t) - f^*\rho_B| \leq h_X(t).
\end{equation*}
Next, we prove for the function $h: [T_{\e_0},T) \to (0,\e_0]$, $h(T_\e) = \e$, that
\begin{equation*}
    h(T_{\e}) \leq CE(T_{\e}),
\end{equation*}
as it follows that $h_X(t) \leq CE(t)$.\\
Recall $T_{\e} = \max\{T^+_{\e},T^-_{\e}, T'_{\e},T''_{\e}\}$. We observe that
\begin{equation*}
    A_{\e}^{+}(t) = \sum_{i=0}^{m-1} E(t)^{m-i}(1-E(t)+\mu_j\e)^i \frac{{n\choose i}\omega_{SPR}^{n-i} \wedge f^* \omega_B^i}{{n\choose m} \omega_{SPR}^{n-m} \wedge f^*\omega_B^m}
\end{equation*}
satisfies 
\begin{equation*}
    A^+_{\e}(t) \leq A_1 E(t)
\end{equation*}
on $X \times [0,T)$ for some uniform $A_1 > 0$ since $0 < \e < 1$ and for all $i = 0, \dots, m-1$,
\begin{equation*}
    -C \omega_{SPR}^{n-m} \wedge f^*\omega_B^m \leq \omega_{SPR}^{n-i} \wedge f^*\omega_B^i \leq C \omega_{SPR}^{n-m} \wedge f^*\omega_B^m.
\end{equation*}
Then we can choose $T^+_{\e}$ by $E(T^+_{\e}) = \frac{\e}{A_1}$ such that $A^+_{\e}(t) \leq \e$ on $X \times [T_{\e}^+,T)$. Similarly,
\begin{equation*}
    A_{\e}^{-}(t) = \sum_{i=0}^{m-1} E(t)^{m-i}(1-E(t)-\mu_j\e)^i \frac{{n\choose i}\omega_{SPR}^{n-i} \wedge f^* \omega_B^i}{{n\choose m} \omega_{SPR}^{n-m} \wedge f^*\omega_B^m}
\end{equation*}
satisfies
\begin{equation*}
    A_{\e}^-(t) \geq -A_2 E(t)
\end{equation*}
on $X \times [0,T)$ for some uniform $A_2 > 0$ since
\begin{align*}
    A_{\e}^{-}(t) & \geq -\sum_{i=0}^{m-1} C_i E(t)^{m-i}|1-E(t)-\mu_j\e|^i\\
    &\geq - \sum_{i=0}^{m-1} C_i E(t)^{m-i} (1+\mu_j)^i\\
    &\geq - A_2 E(t).
\end{align*}
Thus, we can choose $T^-_{\e}$ by $E(T^-_{\e}) = \frac{\e}{A_2}$ such that $A^-_{\e}(t) \geq -\e$ on $X \times [T^-_{\e},T)$.\\
Furthermore, we can choose $T'_{\e}$ by $E(T'_{\e}) = \e$ since
\begin{equation*}
    E(t) \sup_{X} |\rho_{SPR}| \leq 1
\end{equation*}
always holds for all $t \geq T' = T'(\sup_X|\rho_{SPR}|)$ for some $T' \in [0,T)$. Additionally, we choose $T''_{\e}$ by $E(T''_{\e}) = \e$. Therefore, by setting $C = \max\{A_1,A_2,1\}$ we get
\begin{equation*}
    h(T_\e) = \e \leq C E(T_\e)
\end{equation*}
as desired.
\end{proof}

\section{Type I Scalar Curvature}

\noindent We begin with the following Propositions.

\begin{proposition}
Assume that the (\ref{KRF}) satisfies (\ref{VFC}). If $\rho_{SPR} = 0$ on $X \backslash S$, then for any compact set $K \subseteq X \backslash S$ there exists a continuous, strictly decreasing function $h_K:[0,T) \to [0,\infty)$, where $h_K = h_K(n,m,T,\omega_0,f^*\eta, \rho_{SPR}, f^*\rho_B', K, C_K,U_X)$, such that on $K \times [0,T)$ we have
\begin{equation*}
    \left|u(t) - (1-e^{-T}) f^*\rho_B'\right| \leq h_K(t).
\end{equation*}
If (\ref{SKE}) holds, then for any compact set $K \subseteq X \backslash S$ there exists a $h_K:[0,T) \to [0,\infty)$, where $h_K = h_K(n,m,T,\omega_0,f^*\eta, \rho_{SKE}, f^*\rho_B', K, C_K,U_X)$, such that on $K \times [0,T)$ we have
\begin{equation*}
    \left|u(t) - (1-e^{-T}) f^*\rho_B'\right| \leq h_K(t).
\end{equation*}
\end{proposition}

\begin{proof}
In both cases, it is clear that
\begin{align*}
    |u - f^*\rho_B| &= \left|(1-e^{t-T}) \frac{\d \varphi}{\d t} + \varphi - (1-e^{-T})f^*\rho'_B\right|\\
    &\leq  C (1-e^{t-T}) + |\varphi-(1-e^{-T})f^*\rho'_B| \leq h_K(t).
\end{align*}
\end{proof}

\begin{proposition}
Assume that the map $f: X \to B$ is a submersion and the (\ref{KRF}) satisfies (\ref{VFC}). If $\omega_{SPR} = \omega_{0}$ on $X$, then there exists a constant $C = C(n,m,T,\omega_0,f^*\eta, \rho_{SPR}, f^*\rho_B',C_X) > 0$ such that on $X \times [0,T)$ we have
\begin{equation*}
    \left|u(t) - (1-e^{-T}) f^*\rho_B'\right| \leq CE(t).
\end{equation*}
If (\ref{SKE}) holds then there exists a constant $C = C(n,m,T,\omega_0,f^*\eta, \rho_{SKE},f^*\rho_B',C_X) > 0$ such that on $X \times [0,T)$ we have
\begin{equation*}
    \left|u(t) - (1-e^{-T}) f^*\rho_B'\right| \leq CE(t).
\end{equation*}
\end{proposition}

\noindent These two Propositions also hold for the other cases but we will not focus on them in section. Unfortunately, we cannot say any more about the general case where $f: X \to B$ potentially has singular fibres unless we have the control on $h_K$ provided by Lemma \ref{lipschitz_lemma}. So for the remainder of this section we will add the assumption that $h_K$ is Lipschitz so that $h_K(t) \leq CE(t)$ and this will enables us to prove a local Type I scalar curvature bound in any compact set away from singular fibres.\\

\noindent So, in the four settings mentioned in the previous two Propositions and assuming that $h_K$ is Lipschitz, we can fix an $A > 0$ so that
\begin{equation*}
    A E(t) \leq 2AE(t) - (u(t)- (1-e^{-T})f^*\rho'_B) \leq 3AE(t)
\end{equation*}
and denote
\begin{equation*}
    v(t):= 2AE(t) - (u(t)- (1-e^{-T})f^*\rho'_B).
\end{equation*}
It follows from equation (\ref{scalar_curvature_formula}) that 
\begin{equation*}
    (1-e^{t-T}) R_{\omega(t)} = e^{t-T} n - (1-e^{-T})tr_{\omega(t)} f^*\omega'_B + \Delta_{\omega(t)} v.
\end{equation*}
We now prove some Li-Yau type estimates.

\begin{proposition}
Assume that the (\ref{KRF}) satisfies (\ref{VFC}). If $\rho_{SPR} = 0$ on $X \backslash S$ or (\ref{SKE}) holds, then for any compact set $K \subseteq X \backslash S$ where $h_K$ is Lipschitz there exist a constant $C = C(n,m,T,\omega_0,f^*\eta,\rho_{SPR},f^*\rho_B',K,C_K,U_X) > 0$ or a constant $C = C(n,m,T,\omega_0,f^*\eta,\rho_{SKE},f^*\rho_B',K,C_K,U_X) > 0$, respectively, such that on $K \times [0,T)$,
\begin{equation*}
    \frac{|\nabla v|^2}{v} \leq C.
\end{equation*}
On the other hand, assume that $f:X \to B$ is a submersion and (\ref{KRF}) satisfies (\ref{VFC}). If $\omega_{SPR} = \omega_{0}$ on $X$ or if (\ref{SKE}) holds, then there exists a $C = C(n,m,T,\omega_0, f^*\eta,\rho_{SPR},f^*\rho_B',C_X) > 0$ or $C = C(n,m,T,\omega_0,f^*\eta,\rho_{SKE},f^*\rho_B',C_X) > 0$, respectively, such that on $X \times [0,T)$ we have
\begin{equation*}
    \frac{|\nabla v|^2}{v} \leq C.
\end{equation*}
\end{proposition}

\begin{proof}
We remark that in this and the following proof, we shall only provide the argument for the first case where $\rho_{SPR} = 0$ on $X \backslash S$, since the arguments for the remaining cases are similar. If we want to consider the case that $f:X \to B$ is a submersion then we can take $K = X$ and set $\chi = 1$ or $\Upsilon = 1$ on $X$ below, where $\chi$ and $\Upsilon$ are the cut-off functions used in Proposition \ref{trace_lemma}. Recall that $0 \leq \chi \leq 1$ on $X$ and for $K \subset \subset U \subset \subset X \backslash S$
\begin{equation*}
    \chi = 0 \quad \text{on } X \backslash U, \quad \text{and} \quad \chi = 1 \quad \text{on } K,
\end{equation*}
and for some constants $C_1, C_2 > 0$ we set $\Upsilon = e^{-C_1 \chi^{-C_2}}$. Furthermore, choose $K'$ to be a compact set such that $K \subset \subset U \subset \subset K' \subset \subset X \backslash S$.\\ 
First of all, by Bochner's formula, \cite[Prop. 6.4]{B25I}, we have
\begin{equation*}
    \parabolic{}|\nabla v|^2 = |\nabla v|^2 - |\nabla \nabla v|^2 - |\nabla \oo{\nabla} v|^2 + 2 \Real \ip{\nabla \parabolic{}v,\oo{\nabla} v}
\end{equation*}
on $X \backslash S \times [0,T)$, and we further calculate the evolution equation of $\frac{|\nabla v|^2}{v}$ to be
\begin{align*}
    \parabolic{}\frac{|\nabla v|^2}{v} &= \frac{|\nabla v|^2 - |\nabla \nabla v|^2 - |\nabla \oo{\nabla} v|^2}{v}+  \frac{2 \Real\ip{\nabla \parabolic{} v, \oo{\nabla} v}}{v}\\
    &- \frac{|\nabla v|^2}{v^2} \parabolic{} v+ 2 \Real \ip{\nabla \frac{|\nabla v|^2}{v}, \frac{\oo{\nabla} v}{v}}
\end{align*}
and note that
\begin{equation*}
    \parabolic{} v = m-(1-e^{-T})\tr_{\omega(t)} f^*\omega'_B - \frac{2A e^{-t}}{1-e^{-T}}.
\end{equation*}
Additionally, by the argument of \cite[Thm. 4.3]{ST07}, there exists a constant $C > 0$ dependent on the bisectional curvature of $\omega_B$ and the compact set $K'$ such that on $K' \times [0,T)$ we have
\begin{equation*}
    \Delta_{\omega(t)} \tr_{\omega(t)} f^*\omega'_B \geq \ip{\Ric(\omega(t)), f^*\omega'_B} - C (\tr_{\omega(t)} f^*\omega'_B)^2 + \frac{|\nabla\tr_{\omega(t)} f^*\omega'_B|^2}{\tr_{\omega(t)} f^*\omega'_B},
\end{equation*}
and so that the parabolic Schwarz lemma implies $\tr_{\omega(t)} f^*\omega_B' \leq C$ and thus
\begin{align*}
    \parabolic{} \tr_{\omega(t)} f^*\omega'_B \leq C - C^{-1} |\nabla \tr_{\omega(t)} f^*\omega'_B|^2.
\end{align*}
It follows that on $U \times [0,T)$
\begin{align}\label{norm_grad_v_proof_inequality}
    \parabolic{}&\left(\frac{|\nabla v|^2}{v} + (1-e^{-T})\tr_{\omega(t)} f^*\omega'_B\right)\\\notag
    &\leq \frac{|\nabla v|^2 - |\nabla \nabla v|^2 - |\nabla \oo{\nabla} v|^2}{v}- \frac{|\nabla v|^2}{v^2} \left(m-(1-e^{-T})\tr_{\omega(t)} f^*\omega'_B - \frac{2A e^{-t}}{1-e^{-T}}\right)\\\notag
    &+ 2 \Real \ip{\nabla \left(\frac{|\nabla v|^2}{v}+ (1-e^{-T})\tr_{\omega(t)} f^*\omega'_B \right), \frac{\oo{\nabla} v}{v}}\\\notag
    &- 4 \Real \ip{\nabla (1-e^{-T})\tr_{\omega(t)}f^*\omega_B', \frac{\oo{\nabla} v}{v}}\\\notag
    &+ C - C^{-1}|\nabla (1-e^{-T})\tr_{\omega(t)} f^* \omega'_B|^2.\notag
\end{align}
Now, let $\e > 0$ be a small constant to be determined and write
\begin{align*}
    2 \Real &\ip{\nabla \left(\frac{|\nabla v|^2}{v} + (1-e^{-T})\tr_{\omega(t)} f^*\omega'_B \right), \frac{\oo{\nabla} v}{v}}\\ &= 2 \e \Real \ip{\nabla \left(\frac{|\nabla v|^2}{v} + (1-e^{-T})\tr_{\omega(t)} f^*\omega'_B \right), \frac{\oo{\nabla} v}{v}}\\
    &+ 2(1-\e)\Real \ip{\nabla \left(\frac{|\nabla v|^2}{v} + (1-e^{-T})\tr_{\omega(t)} f^*\omega'_B \right), \frac{\oo{\nabla} v}{v}}.
\end{align*}
It follows that for some $\e' > 0$ to be determined,
\begin{align*}
    2 \e \Real &\ip{\nabla \left(\frac{|\nabla v|^2}{v} + (1-e^{-T})\tr_{\omega(t)} f^*\omega'_B \right), \frac{\oo{\nabla} v}{v}}\\ &\leq C \frac{|\nabla v|^2}{v^2} + \e' |\nabla (1-e^{-T})\tr_{\omega(t)} f^*\omega'_B|^2 -2 \e \frac{|\nabla v|^4}{v^3}+ \frac{2\e \Real \ip{\nabla |\nabla v|^2, \oo{\nabla} v}}{v^2}.
\end{align*}
Then we use the inequality $\big|\ip{\nabla |\nabla v|^2, \oo{\nabla} v}\big| \leq |\nabla v|^2 (|\nabla \nabla v| + |\nabla \oo{\nabla} v|)$ to find
\begin{align*}
    \frac{2\e \Real \ip{\nabla |\nabla v|^2, \oo{\nabla} v}}{v^2} &\leq \frac{2\e |\nabla v|^2 (|\nabla \nabla v| + |\nabla \oo{\nabla} v|)}{v^2} \\
    &\leq \frac{\e^2|\nabla v|^4}{v^3} + \frac{|\nabla \nabla v|^2 + |\nabla \oo{\nabla} v|^2}{v},
\end{align*}
and finally, for the same $\e' > 0$, we have
\begin{equation*}
    - 4 \Real \ip{\nabla (1-e^{-T})\tr_{\omega(t)}f^*\omega_B', \frac{\oo{\nabla} v}{v}} \leq \e'|\nabla(1-e^{-T})\tr_{\omega(t)} f^*\omega_B'|^2 + C \frac{|\nabla v|^2}{v^2}.
\end{equation*}
Combining the estimates so far, we have on $U \times [0,T)$
\begin{align*}
    \parabolic{} &\left(\frac{|\nabla v|^2}{v} + (1-e^{-T})\tr_{\omega(t)} f^*\omega'_B\right)\\
    &\leq C + C\frac{|\nabla v|^2}{v} +C \frac{|\nabla v|^2}{v^2} - (2\e-\e^2)\frac{|\nabla v|^4}{v^3}\\
    &+ 2 (1-\e)\Real \ip{\nabla \left(\frac{|\nabla v|^2}{v}+ (1-e^{-t})\tr_{\omega(t)} f^*\omega'_B \right), \frac{\oo{\nabla} v}{v}}\\
    & + (2\e' - C^{-1}) |\nabla (1-e^{-T})\tr_{\omega(t)} f^*\omega'_B|^2\\
    &\leq C +C \frac{|\nabla v|^2}{v^2} - (2\e-\e^2)\frac{|\nabla v|^4}{v^3}\\
    &+ 2 (1-\e)\Real \ip{\nabla \left(\frac{|\nabla v|^2}{v}+ (1-e^{-T})\tr_{\omega(t)} f^*\omega'_B \right), \frac{\oo{\nabla} v}{v}}
\end{align*}
where we choose $\e'$ sufficiently small such that $(2\e' - C^{-1})|\nabla (1-e^{-T})\tr_{\omega(t)} f^*\omega'_B|^2 \leq 0$, apply $0 < v \leq C$, and use $0 \leq \tr_{\omega(t)} f^*\omega_B \leq C$ on $K' \times [0,T)$. We now set 
\begin{equation*}
    Q_3 = \frac{|\nabla v|^2}{v} + (1-e^{-T})\tr_{\omega(t)} f^*\omega'_B
\end{equation*}
so that on $X \times [0,T)$ we have
\begin{equation*}
    \parabolic{} \Upsilon Q_3 = \Upsilon \parabolic{} Q_3 +\left(\frac{|\nabla \Upsilon|^2}{\Upsilon} -\Delta \Upsilon\right)Q_3 - 2 \Real \ip{\nabla (\Upsilon Q_3), \frac{\oo{\nabla} \Upsilon}{\Upsilon}},
\end{equation*}
where $\Upsilon = e^{-C_1\chi^{-C_2}}$ as before. We estimate the middle term by
\begin{align*}
    \left(\frac{|\nabla \Upsilon|^2}{\Upsilon} - \Delta \Upsilon \right) Q_3 \leq  C \chi^{-2C_2-2} \Upsilon Q_3 \leq C\chi^{-2C_2-2}\Upsilon + C\frac{|\nabla v|^2}{v^2} \chi^{-2C_2-2}\Upsilon.
\end{align*}
We also note that within the first term we need the following estimate
\begin{align*}
    2(1-\e) \Upsilon \Real \ip{{\nabla Q_3}, \frac{\oo{\nabla} v}{v}} &= 2(1-\e) \Real \ip{{\nabla (\Upsilon Q_3)}, \frac{\oo{\nabla} v}{v}}\\
    &-2(1-\e) Q_3 \Real \ip{{\nabla \Upsilon}, \frac{\oo{\nabla} v}{v}}\\
    &\leq 2(1-\e) \Real \ip{{\nabla (\Upsilon Q_3)}, \frac{\oo{\nabla} v}{v}}\\
    &+ \left(\frac{C|\nabla \Upsilon|^2}{\Upsilon} + \frac{\e^2 |\nabla v|^2}{v^2} \Upsilon \right) Q_3\\
    &\leq 2(1-\e) \Real \ip{\nabla (\Upsilon Q_3), \frac{\oo{\nabla} v}{v}}\\
    &+ C \chi^{-2C_2-2}\Upsilon + C \frac{|\nabla v|^2}{ v^2} \chi^{-2C_2-2}\Upsilon + \frac{\e^2|\nabla v|^4}{v^3} \Upsilon.
\end{align*}
Therefore, we find that all together,
\begin{align*}
    \parabolic{} \Upsilon Q_3 &\leq C \chi^{-2C_2-2}\Upsilon + C \frac{|\nabla v|^2}{v^2} \chi^{-2C_2-2}\Upsilon -(2\e-2\e^2) \frac{|\nabla v|^4}{v^3} \Upsilon\\
    &+ 2(1-\e) \Real \ip{\nabla(\Upsilon Q_3), \frac{\oo{\nabla} v}{v}} - 2 \Real \ip{\nabla (\Upsilon Q_3), \frac{\oo{\nabla} \Upsilon}{\Upsilon}}\\
    &\leq\ \frac{C \chi^{-4C_2-4}\Upsilon}{v} -(2\e-3\e^2) \frac{|\nabla v|^4}{v^3} \Upsilon\\
    &+ 2(1-\e) \Real \ip{\nabla(\Upsilon Q_3), \frac{\oo{\nabla} v}{v}} - 2 \Real \ip{\nabla (\Upsilon Q_3), \frac{\oo{\nabla} \Upsilon}{\Upsilon}}
\end{align*}
and we choose $\e = \frac{1}{2}$ so that $2\e - 3\e^2 = \frac{1}{4}$.\\
Now let $(x^*,t^*) \in X \times [0,T)$ be a point at which $\Upsilon Q_3$ achieves a maximum. If $x^* \in X \backslash U$ then $(\Upsilon Q_3)(x^*,t^*) = 0$ is bounded. If $x^* \in U$ then $v(x^*,t^*) > 0$ and 
\begin{equation*}
    0 \leq \frac{C\chi^{-4C_2-4}(x^*)\Upsilon(x^*)}{v(x^*,t^*)} - \frac{1}{4} \frac{|\nabla v|^4(x^*,t^*)}{v(x^*,t^*)^3} \Upsilon(x^*)
\end{equation*}
so that
\begin{equation*}
    \frac{|\nabla v|^2}{v}(x^*,t^*) \Upsilon(x^*) \leq C
\end{equation*}
and therefore 
\begin{equation*}
    (\Upsilon Q_3)(x^*,t^*) \leq C.
\end{equation*}
It follows that for all $(x,t) \in X \times [0,T)$ we have
\begin{equation*}
    (\Upsilon Q_3)(x,t) \leq C
\end{equation*}
so that
\begin{equation*}
    \frac{|\nabla v|^2}{v} \leq C\Upsilon^{-1},
\end{equation*}
and the result is immediate.
\end{proof}

\noindent As a consequence, for any compact $K \subseteq X \backslash S$ there exists a constant $C > 0$ such that 
\begin{equation*}
    |\nabla (u-f^*\rho_B)|^2 \leq CE(t)
\end{equation*}
on $K \times [0,T)$.

\begin{proposition}
Assume that the (\ref{KRF}) satisfies (\ref{VFC}). If $\rho_{SPR} = 0$ on $X \backslash S$ or (\ref{SKE}) holds, then for any compact set $K \subseteq X \backslash S$ where $h_K$ is Lipschitz there exist a constant $C = C(n,m,T,\omega_0,f^*\eta,\rho_{SPR},f^*\rho_B',K,C_K,U_X) > 0$ or a constant $C = C(n,m,T,\omega_0,f^*\eta,\rho_{SKE},f^*\rho_B',K,C_K,U_X) > 0$, respectively, such that on $K \times [0,T)$,
\begin{equation*}
    \Delta_{\omega(t)} v \leq C.
\end{equation*}
On the other hand, assume that $f:X \to B$ is a submersion and (\ref{KRF}) satisfies (\ref{VFC}). If $\omega_{SPR} = \omega_{0}$ on $X$ or if (\ref{SKE}) holds, then there exists a $C = C(n,m,T,\omega_0, f^*\eta,\rho_{SPR},f^*\rho_B',C_X) > 0$ or $C = C(n,m,T,\omega_0,f^*\eta,\rho_{SKE},f^*\rho_B',C_X) > 0$, respectively, such that on $X \times [0,T)$ we have
\begin{equation*}
    \Delta_{\omega(t)} v \leq C.
\end{equation*}
\end{proposition}

\begin{proof}
Again, we only provide the argument for the first case where $\rho_{SPR} = 0$ on $X \backslash S$ and take $K \subset \subset U \subset \subset K' \subset \subset X \backslash S$. We can calculate the evolution equation of $\Delta v - (1-e^{-T})\tr_{\omega(t)} f^*\omega'_B$ over $X \backslash S \times [0,T)$ to be 
\begin{align*}
    \parabolic{}&(\Delta v - (1-e^{-T})\tr_{\omega(t)} f^*\omega'_B)\\ &= \Delta v - (1-e^{-T})\tr_{\omega(t)}f^*\omega'_B+ \ip{\Ric, -(1-e^{-T})f^* \omega'_B + i\d\dbar v}\\
    &= \Delta v - (1-e^{-T})\tr_{\omega(t)} f^*\omega'_B+\frac{e^{t-T}}{1-e^{t-T}} (\Delta v - (1-e^{-T})\tr_{\omega(t)} f^*\omega'_B )\\
    &+ \frac{1}{1-e^{t-T}}\left(|(1-e^{-T})f^*\omega'_B|^2 + |\nabla \oo{\nabla} v|^2 + 2 \Real \ip{(1-e^{-T})f^*\omega'_B, i\d\dbar v }\right).
\end{align*}
where the last line uses
\begin{equation*}
    (1-e^{t-T}) \Ric = e^{t-T} \omega(t) - (1-e^{-T})f^*\omega'_B + i\d\dbar v.
\end{equation*}
It follows that on $U \times [0,T)$ we have 
\begin{align*}
    \parabolic{} &\left(\frac{(1-e^{t-T})(\Delta v - (1-e^{-T})\tr_{\omega(t)} f^*\omega'_B)}{v}\right)\\
    &= (1-e^{t-T}) \left(\frac{\parabolic{} (\Delta v - (1-e^{-T})\tr_{\omega(t)} f^*\omega'_B)}{v}\right)\\
    &-e^{t-T} \frac{\Delta v - (1-e^{-T}) \tr_{\omega(t)}f^*\omega'_B}{ v}\\
    &- (1-e^{t-T})\left(\frac{\Delta v - (1-e^{-T})\tr_{\omega(t)} f^*\omega'_B}{v} \frac{\parabolic{} v}{v}\right)\\
    &+ 2\Real \ip{\nabla \left(\frac{(1-e^{t-T})(\Delta v - (1-e^{-T})\tr_{\omega(t)} f^*\omega'_B)}{v}\right), \frac{\oo{\nabla} v}{v}}\\
    &= (1-e^{t-T}) \frac{\Delta v - (1-e^{-T})\tr_{\omega(t)} f^*\omega'_B}{v} \left(1+\frac{-m+(1-e^{-T})\tr_{\omega(t)}f^*\omega'_B + \frac{2A e^{-t}}{1-e^{-T}}}{v}\right)\\
    &+\frac{|(1-e^{-T})f^*\omega'_B|^2 + |\nabla \oo{\nabla} v|^2 + 2 \Real \ip{(1-e^{-T})f^*\omega'_B, i\d\dbar v }}{v}\\
    &+ 2\Real \ip{\nabla \left(\frac{(1-e^{t-T})(\Delta v - (1-e^{-T})\tr_{\omega(t)} f^*\omega'_B)}{v}\right), \frac{\oo{\nabla} v}{v}}\\
    &\leq \frac{C}{v} + (1-e^{t-T}) \frac{\Delta v }{v} \left(1+\frac{-m+(1-e^{-T})\tr_{\omega(t)}f^*\omega'_B + \frac{2A e^{-t}}{1-e^{-T}}}{v}\right)+\frac{2|\nabla \oo{\nabla} v|^2}{v}\\
    &+ 2\Real \ip{\nabla \left(\frac{(1-e^{t-T})(\Delta v - (1-e^{-T})\tr_{\omega(t)} f^*\omega'_B)}{v}\right), \frac{\oo{\nabla} v}{v}}
\end{align*}
where we apply the bounds $|f^*\omega_B'|_{\omega(t)}^2 \leq (\tr_{\omega(t)}f^*\omega_B')^2$ and $0 \leq \tr_{\omega(t)} f^*\omega'_B \leq C$, on $K' \times [0,T)$. Recall inequality (\ref{norm_grad_v_proof_inequality}) from the previous Proposition so that over $U \times [0,T)$ we have
\begin{align*}
    \parabolic{}&\left(\frac{|\nabla v|^2}{v} + (1-e^{-T})\tr_{\omega(t)} f^*\omega'_B\right)\\
    &\leq \frac{|\nabla v|^2 - |\nabla \nabla v|^2 - |\nabla \oo{\nabla} v|^2}{v}+ C\frac{|\nabla v|^2}{v^2} + C\\
    &+ 2 \Real \ip{\nabla \left(\frac{|\nabla v|^2}{v}+ (1-e^{-T})\tr_{\omega(t)} f^*\omega'_B \right), \frac{\oo{\nabla} v}{v}},
\end{align*}
so that after setting 
\begin{equation*}
    Q_4 = \frac{(1-e^{t-T})(\Delta v - (1-e^{-T})\tr_{\omega(t)} f^*\omega'_B)}{v}, \quad Q_5 = Q_4 + 3 Q_3,
\end{equation*}
we have
\begin{align*}
    \parabolic{}Q_5 &\leq C + \frac{C}{v}-\frac{|\nabla \oo{\nabla} v|^2}{v} + 2\Real \ip{\nabla Q_5, \frac{\oo{\nabla} v}{v}}\\
    &+(1-e^{t-T}) \frac{\Delta v }{v} \left(1 + \frac{-m+(1-e^{-T})\tr_{\omega(t)}f^*\omega'_B + \frac{2A e^{-t}}{1-e^{-T}}}{v}\right)
\end{align*}
on $U \times [0,T)$. Next, observe that $n|\nabla \oo{\nabla} v|^2 \geq (\Delta v)^2$ so that
\begin{align*}
    \parabolic{} Q_5 &\leq \frac{C}{v} -\frac{\frac{1}{n}(\Delta v)^2}{v} + 2\Real \ip{\nabla Q_5, \frac{\oo{\nabla} v}{v}}\\
    &+ (1-e^{t-T}) \frac{\Delta v }{v} \left(1+\frac{-m+(1-e^{-T})\tr_{\omega(t)}f^*\omega'_B + \frac{2A e^{-t}}{1-e^{-T}}}{v}\right)\\
    &\leq \frac{C}{v} + \frac{(\e-\frac{1}{n})(\Delta v)^2}{v} + 2 \Real \ip{\nabla Q_5, \frac{\oo{\nabla} v}{v}}
\end{align*}
for some $\e > 0$ to be fixed. As before, taking $\Upsilon = e^{-C_1\chi^{-C_2}}$, we find that on $X \times [0,T)$,
\begin{equation*}
    \parabolic{} \Upsilon Q_5 = \Upsilon \parabolic{} Q_5 + \left(\frac{|\nabla \Upsilon|^2}{\Upsilon} - \Delta \Upsilon\right) Q_5 - 2 \Real \ip{\nabla(\Upsilon Q_5), \frac{\oo{\nabla} \Upsilon}{\Upsilon}}.
\end{equation*}
We estimate the middle term by
\begin{equation*}
    \left(\frac{|\nabla \Upsilon|^2}{\Upsilon} - \Delta \Upsilon\right) Q_5 \leq C\chi^{-2C_2-2}\Upsilon |Q_5| \leq C\chi^{-4C_2-4}\Upsilon + \frac{\e(\Delta v)^2}{v}\Upsilon.
\end{equation*}
We also note that within the first term we need the following estimate
\begin{align*}
    2 \Upsilon \Real \ip{\nabla Q_5, \frac{\oo{\nabla}v}{v}} &= 2 \Real \ip{\nabla (\Upsilon Q_5), \frac{\oo{\nabla}v}{v}} - 2Q_5 \Real \ip{\nabla \Upsilon, \frac{\oo{\nabla}v}{v}}\\
    &\leq 2 \Real \ip{\nabla (\Upsilon Q_5), \frac{\oo{\nabla}v}{v}} + \frac{C|Q_5|}{\sqrt{v}} \chi^{-2C_2-2} \Upsilon\\
    &\leq 2 \Real \ip{\nabla(\Upsilon Q_5), \frac{\oo{\nabla}v}{v}} + C\chi^{-4C_2-4}\Upsilon + \frac{\e(\Delta v)^2}{v}\Upsilon.
\end{align*}
All together, we currently have
\begin{align*}
    \parabolic{} \Upsilon Q_5 &\leq C\chi^{-4C_2-4}\Upsilon+ \frac{C}{v}\Upsilon + \frac{(3\e-\frac{1}{n})(\Delta v)^2}{v}\Upsilon\\
    &+ 2 \Real \ip{\nabla(\Upsilon Q_5), \frac{\oo{\nabla}v}{v}}-2 \Real \ip{\nabla(\Upsilon Q_5), \frac{\nabla \Upsilon}{\Upsilon}}.
\end{align*}
Now fix $\e = \frac{1}{6n}$ so that $3\e - \frac{1}{n} = -\frac{1}{2n}$. Assume that $(x^*,t^*) \in X \times [0,T)$ is a point at which $\Upsilon Q_5$ achieves a maximum, if $x^* \in X \backslash U$ then $(\Upsilon Q_5)(x^*,t^*) = 0$ is bounded. Otherwise, $x^* \in U$ and
\begin{equation*}
    0 \leq  C \chi(x^*)^{-4C_2-4}\Upsilon(x^*)+ \frac{C}{v(x^*,t^*)}\Upsilon(x^*) - \frac{\frac{1}{2n}(\Delta v)(x^*,t^*)^2}{v(x^*,t^*)}\Upsilon(x^*)
\end{equation*}
so that
\begin{equation*}
    (\Delta v)(x^*,t^*) \leq C
\end{equation*}
since $0 < v(x^*,t^*) \leq C$, and we see
\begin{equation*}
    (\Upsilon Q_5)(x^*,t^*) \leq C.
\end{equation*}
It follows that for all $(x,t) \in X \times [0,T)$ we have
\begin{equation*}
    (\Upsilon Q_5)(x,t) \leq C
\end{equation*}
so that $\Upsilon Q_4 \leq C$ and thus
\begin{equation*}
    \Delta v \leq \frac{C\Upsilon^{-1}v}{(1-e^{t-T})} + (1-e^{-T})\tr_{\omega(t)} f^*\omega'_B.
\end{equation*}
The result follows.
\end{proof}

\noindent On the other hand, we note that since the scalar curvature is uniformly bounded below,
\begin{equation*}
    -C(1-e^{t-T}) \leq (1-e^{t-T}) R = e^{t-T}n - (1-e^{-T})\tr_{\omega(t)} f^*\omega'_B + \Delta_{\omega(t)} v \leq n + \Delta_{\omega(t)} v
\end{equation*}
so that $\Delta_{\omega(t)} v$ is also uniformly bounded below on $K \times [0,T)$ for any compact set $K \subseteq X \backslash S$. As for the Ricci potential $u$, we see that on $K \times [0,T)$
\begin{equation*}
    |\Delta_{\omega(t)} u| \leq |\Delta_{\omega(t)} v| + (1-e^{-T})|\Delta_{\omega(t)} f^*\rho'_B| \leq C,
\end{equation*}
where we apply the parabolic Schwarz lemma to $\alpha = i\d\dbar (\chi \rho'_B)$ for a suitable cut-off function $\chi$ on $B$. Additionally, one can extend these $C^0$-estimates to the function $|\d_t u|$ since the right hand side of
\begin{equation*}
    \parabolic{}u = -m + \tr_{\omega(t)} f^*\eta,
\end{equation*}
is globally bounded on $X \times [0,T)$.\\

\noindent Type I scalar curvature is now immediate.

\begin{theorem}
Assume that the (\ref{KRF}) satisfies (\ref{VFC}). If $\rho_{SPR} = 0$ on $X \backslash S$ or (\ref{SKE}) holds, then for any compact set $K \subseteq X \backslash S$ where $h_K$ is Lipschitz there exist a constant $C = C(n,m,T,\omega_0,f^*\eta,\rho_{SPR},f^*\rho_B',K,C_K,U_X) > 0$ or a constant $C = C(n,m,T,\omega_0,f^*\eta,\rho_{SKE},f^*\rho_B',K,C_K,U_X) > 0$, respectively, such that on $K \times [0,T)$,
\begin{equation*}
    R_{\omega(t)} \leq \frac{C}{1-e^{t-T}}.
\end{equation*}
On the other hand, assume that $f:X \to B$ is a submersion and (\ref{KRF}) satisfies (\ref{VFC}). If $\omega_{SPR} = \omega_{0}$ on $X$ or if (\ref{SKE}) holds, then there exists a $C = C(n,m,T,\omega_0, f^*\eta,\rho_{SPR},f^*\rho_B',C_X) > 0$ or $C = C(n,m,T,\omega_0,f^*\eta,\rho_{SKE},f^*\rho_B',C_X) > 0$, respectively, such that on $X \times [0,T)$ we have
\begin{equation*}
    R_{\omega(t)} \leq \frac{C}{1-e^{t-T}}.
\end{equation*}
\end{theorem}

\begin{proof}
We have
\begin{equation*}
    (1-e^{t-T}) R_{\omega(t)} = e^{t-T}n - (1-e^{-T})\tr_{\omega(t)} f^*\omega'_B + \Delta_{\omega(t)} v \leq C.
\end{equation*}
\end{proof}

\noindent To end, we make a comparison between the infinite time case and the finite time case. Given that $K_X$ is semi-ample and induces a map into projective space, it was shown by Song and Tian, \cite{ST16}, that on $X \times [0,\infty)$ we have
\begin{equation*}
    R_{\omega(t)} \leq C,
\end{equation*}
while it was shown by Jian, \cite{J20}, that for every compact $K \subseteq X \backslash S$, we have on $K \times [0, \infty)$,
\begin{equation*}
    \lim_{t \to \infty} R_{\omega(t)} \to -m.
\end{equation*}
Returning to our setting, it was shown by Zhang, \cite{Z10}, that on $X \times [0,T)$ we have
\begin{equation*}
    R_{\omega(t)} \leq \frac{C}{(T-t)^2},
\end{equation*}
while we have proven that, under certain additional assumptions, for every compact $K \subseteq X \backslash S$ we have on $K \times [0,T)$ that
\begin{equation*}
    R_{\omega(t)} \leq \frac{C}{T-t}.
\end{equation*}
The method to use the convergence of the Ricci potential to the potentials of the twisted K\"ahler-Einstein metrics in order to find estimates for the Li-Yau type quantities is the common approach between Jian's result and ours. We expect that (\ref{VFC}) holds in general, but it is interesting that we need the extra condition that the fibres generically admit a smoothly varying family of K\"ahler-Einstein metrics. This is a consequence of Theorem \ref{C0_varphi_estimate} so to extend the Type I scalar curvature estimate to the general setting either an improvement on Theorem \ref{C0_varphi_estimate} is needed or, more likely, a new approach. Additionally, for the case of singular fibres we need that the convergence rate of the Ricci potential to the potentials of the twisted K\"ahler-Einstein metrics is Lipschitz, whereas this technicality does not appear in Jian's result thanks to \cite[Lem. 4.1]{J20} which states that any such $h_K(t)$ can be bounded above by a smooth, decreasing function $A(t)$ such that $0 \leq -A(t)' \leq 100A(t)$ if $T = \infty$.


\begin{thebibliography}{99}

\bibitem{B25II} Bednarek, A., (2025),
\textit{Fano Fibrations and Twisted K\"ahler-Einstein Metrics I},
Preprint

\bibitem{B25I} Bednarek, A., (2025), 
\textit{Global Ricci Curvature Behaviour in the K\"ahler-Ricci Flow with Finite Time Singularities}, 
Advances in Mathematics, \textbf{480}A, 110465

\bibitem{BEG13} Boucksom, S., Eyssidieux, P., Guedj, V., (2013), 
\textit{An Introduction to the K\"ahler Ricci Flow}, 
Springer

\bibitem{C85} Cao, H., (1985), 
\textit{Deformation of Kaehler metrics to Kaehler-Einstein metrics on compact Kaehler manifolds}, 
Invent. Math. \textbf{81}(2), 359-372

\bibitem{CW16} Chen, X., Wang, B., (2016), 
\textit{Space of Ricci Flows II}, 
arXiv:1405.6797v4

\bibitem{CHM25} Conlon, R., Hallgren, M., Ma, Z., (2025), 
\textit{Non-Collapsed Finite Time Singularities of the Ricci Flow on Compact K\"ahler Surfaces are of Type I}, 
arXiv:2502.19804v1

\bibitem{DP10} Demailly, J., Pali, N., (2010), 
\textit{Degenerate Complex Monge Ampere Equations over Compact K\"ahler Manifolds}, 
International Journal of Mathematics \textbf{21}(3), 357-405

\bibitem{EGZ08} Eyssidieux, P., Guedj, V., Zeriahi, A., (2008), 
\textit{A Priori $L^{\infty}$-Estimates for Degenerate Complex Monge-Ampere Equations}, International Mathematics Research Notices, \textbf{2008}(9)

\bibitem{F15} Fong, F., (2015), 
\textit{On the Collapsing Rate of the K\"ahler-Ricci Flow With Finite-Time Singularity}, 
J. Geom. Anal. \textbf{25}, 1098-1107

\bibitem{FZ12} Fong, F., Zhang, Z., (2012), 
\textit{The Collapsing Rate of the K\"ahler-Ricci Flow with Regular Infinite Time Singularity}, 
Crelle's Journal 2015, 95-113

\bibitem{FZ16} Fu, X., Zhang, S., (2016), 
\textit{The K\"ahler-Ricci Flow on Fano Bundles},
Mathematische Zeitschrift, \textbf{286}, 1605-1626

\bibitem{GPSS24} Guo B., Phong, D., Song, J., Sturm, J., (2024), \textit{Diameter Estimates in K\"ahler Geometry}, 
Comm. Pure Appl. Math. \textbf{77}(8), 3520-3556

\bibitem{HJST23} Hallgren, M., Jian, W., Song, J., Tian, G., (2023), \textit{Geometric Regularity of Blow-up Limits of the K\"ahler-Ricci Flow}, 
arXiv:2310.08610v1

\bibitem{HLT24} Heins, H., Lee, M., Tosatti, V., (2024), 
\textit{Collapsing Immortal K\"ahler-Ricci Flows}, 
arXiv:2405.04208v1

\bibitem{J20} Jian, W., (2020), 
\textit{Convergence of Scalar Curvature of K\"ahler-Ricci Flow on Manifolds of Positive Kodaira Dimension}, Advances in Mathematics \textbf{371}, 107253

\bibitem{JST23} Jian, W., Song, J., Tian, G., (2023), 
\textit{Finite Time Singularities of the K\"ahler-Ricci Flow}, arXiv:2310.07945v2

\bibitem{LY86} Li, P., Yau, S., (1986), 
\textit{On the Parabolic Kernel of the Schr\"odinger Operator}, 
Acta Math. \textbf{156}, 153-201

\bibitem{MT25} Miao, M., Tian, G., (2025), 
\textit{A Note on K\"ahler-Ricci Flow on Fano Threefolds}, 
Peking Math Journal \textbf{8}(1), 191-199

\bibitem{PSS12} Phong, D., Song, J., Sturm, J., (2012), 
\textit{Complex Monge-Ampere Equations}, 
Surveys in Differential Geometry XVII, 327-410

\bibitem{ST08} Sesum, N., Tian, G., (2008), 
\textit{Bounding Scalar Curvature and Diameter Along the K\"ahler-Ricci Flow (After Perelman)}, 
Journal of the Institute of Mathematics of Jussieu, \textbf{7}(3), 575-587

\bibitem{S21} Shen, X., (2021), 
\textit{The K\"ahler-Ricci Flow, Holomorphic Vector Fields and Fano Bundles}, Transactions of the AMS, \textbf{374}(9)

\bibitem{SSW12} Song, J., Szekelyhidi, G., Weinkove, B., (2012), 
\textit{The K\"ahler-Ricci Flow on Projective Bundles},
International Mathematics Research Notices \textbf{2013}(2), 243-257

\bibitem{ST07} Song, J., Tian, G., (2007), 
\textit{The K\"ahler-Ricci Flow on Surfaces of Positive Kodaira Dimension}, 
Invent. Math. \textbf{170}, 609-653

\bibitem{ST12} Song, J., Tian, G., (2012), 
\textit{Canonical Measures and K\"ahler-Ricci Flow}, 
Journal of the American Mathematical Society \textbf{25}(2), 303-353

\bibitem{ST16} Song, J., Tian, G., (2016),
\textit{Bounding Scalar Curvature for Global Solutions of the K\"ahler-Ricci Flow}, American Journal of Mathematics \textbf{138}(3), 683-695

\bibitem{SW11} Song, J., Weinkove, B., (2011), 
\textit{The K\"ahler-Ricci Flow on Hirzebruch Surfaces}, 
Journal fur die Reine und Angewandte Mathematik \textbf{659}, 141-168

\bibitem{SZ25} Sun, S., Zhang, J., (2025),
\textit{K\"ahler-Ricci Shrinkers and Fano Fibrations},
arXiv:2410.09661v2 

\bibitem{TZ06} Tian, G., Zhang, Z., (2006),
\textit{On the K\"ahler-Ricci Flow on Projective Manifolds of General Type}, 
Chin. Ann. Math. Ser. B \textbf{27}(2), 179-192

\bibitem{TZ15} Tian, G., Zhang, Z., (2015), 
\textit{Convergence of K\"ahler-Ricci Flow on Lower-Dimensional Algebraic Manifolds of General Type}, 
International Mathematics Research Notices \textbf{2016}(21), 6493-6511

\bibitem{TZ16} Tian, G., Zhang, Z., (2016), 
\textit{Regularity of K\"ahler-Ricci flows on Fano manifolds},
Acta. Math. \textbf{216}, 127-176 

\bibitem{T07} Topping, P., (2007), 
\textit{Relating Diameter and Mean Curvature for Submanifolds of Euclidean Space},
Comment. Math. Helv. \textbf{83}, 539-546

\bibitem{T10} Tosatti, V., (2010), 
\textit{Adiabatic Limits of Ricci-Flat K\"ahler Metrics}, 
J. Differential Geom. \textbf{84}(2), 427-453

\bibitem{T19} Tosatti, V., (2019), 
\textit{KAWA Lecture Notes on the K\"ahler-Ricci Flow}, 
arXiv:1508.04823v5

\bibitem{T24} Tosatti, V., (2024), 
\textit{Immortal Solutions of the K\"ahler-Ricci Flow},  arXiv:2405.04444v2

\bibitem{TWY18} Tosatti, V., Weinkove, B., Yang, X., (2018), 
\textit{The K\"ahler-Ricci Flow, Ricci Flat Metrics and Collapsing Limits}, 
American Journal of Mathematics \textbf{140}(3), 653-698

\bibitem{WZ21} Wang, F., Zhu, X., (2021), 
\textit{Tian’s Partial $C^0$-Estimate Implies the Hamilton-Tian’s Conjecture}, 
Advances in Mathematics \textbf{381}, 107619

\bibitem{Y78} Yau, S., (1978), 
\textit{On The Ricci Curvature of a Compact K\"ahler Manifold and the Complex Monge-Ampere Equation I}, 
Comm. Pure Appl. Math. \textbf{31}(3), 339-411

\bibitem{ZZ23} Zhang, L., Zhang, Z., (2023), 
\textit{On the Finite Time Collapsing Rate of K\"ahler-Ricci Flow on Projective Bundles}, 
Proceedings of the American Mathematical Society, \textbf{151}(12), 5385-5389

\bibitem{ZZ20} Zhang, Y., Zhang, Z., (2020), 
\textit{The Continuity Method on Fano Fibrations}, 
International Mathematics Research Notices \textbf{2020}(22), 8697-8728

\bibitem{Z09} Zhang, Z., (2009), 
\textit{Scalar Curvature Bound for K\"ahler-Ricci Flows over Minimal Manifolds of General Type}, 
International Mathematics Research Notices \textbf{2009}(20), 3901-3912

\bibitem{Z10} Zhang, Z., (2010), 
\textit{Scalar Curvature Behaviour For Finite Time Singularity of K\"ahler Ricci Flow}, 
Michigan Math. J. \textbf{59}(2), 419-433

\end{thebibliography}
\end{document}